\documentclass[10pt]{amsart}
\usepackage{amsmath, amsfonts, amscd, latexsym, amsthm, amssymb}
\usepackage{hyperref}

\def \c{\mathbb{C}}
\def \z{\mathbb{Z}}
\def \r{\mathbb{R}}
\def \n{\mathbb{N}}
\def \p{\mathbb{P}}

\def \ol{\overline}
\def \k{{\bf k}}

\def \R{\mathcal{R}}
\def \L{\mathcal{L}}
\def \K{{\bf K}_{rat}(X)}

\def \.{\cdot}
\def \Con{\textup{Con}}
\def \Reg{\textup{Reg}}

\def \ind{\textup{ind}}

\def \Vol{\textup{Vol}}

\def \supp{\textup{supp}}

\def \ker{\textup{ker}}

\def \span{\textup{span}}

\def \ord{\textup{ord}}

\def \ratmap{\dashrightarrow}

\usepackage[margin=1.4in]{geometry}

\theoremstyle{plain}
\newtheorem{Th}{Theorem}[section]
\newtheorem{Lem}[Th]{Lemma}
\newtheorem{Prop}[Th]{Proposition}
\newtheorem{Cor}[Th]{Corollary}

\newtheorem{THM}{Theorem}

\theoremstyle{definition}
\newtheorem{Ex}[Th]{Example}
\newtheorem{Def}[Th]{Definition}
\newtheorem{Rem}[Th]{Remark}

\begin{document}
\title[Convex bodies, semigroups of integral points and graded algebras]
{Newton-Okounkov bodies, semigroups of integral points, graded algebras and intersection theory}

\author{Kiumars Kaveh, A. G. Khovanskii}

\thanks{The second author is partially supported by Canadian Grant {\rm N~156833-02}.}

\maketitle

\begin{center}
{\it To the memory of Vladimir Igorevich Arnold}
\end{center}

\begin{abstract}
Generalizing the notion of Newton polytope,
we define the Newton-Okounkov body, respectively, for semigroups of integral
points, graded algebras, and linear series on varieties.
We prove that any semigroup in the lattice $\z^n$ is asymptotically approximated
by the semigroup of all the points in a sublattice and
lying in a convex cone. Applying this we obtain several results:
we show that for a large class of graded algebras, the Hilbert functions have polynomial growth
and their growth coefficients satisfy a Brunn-Minkowski type inequality. We prove analogues of Fujita approximation theorem for semigroups of integral points and graded algebras, which implies a generalization of this theorem for arbitrary linear series. Applications to intersection theory include a far-reaching generalization of the Kushnirenko theorem (from Newton polytope theory) and a new version of
the Hodge inequality. We also give elementary proofs of the Alexandrov-Fenchel inequality
in convex geometry and its analogue in algebraic geometry.\\
\end{abstract}

\noindent {\it Key words:} semigroup of integral points, convex body, mixed volume,
Alexandrov-Fenchel inequality, Hilbert function, graded algebra, Cartier divisor and linear series,
Hodge index theorem, Bernstein-Kushnirenko theorem\\

\noindent{\it AMS subject classification:} 14C20, 13D40, 52A39\\

\tableofcontents
\section*{Introduction}
This paper is dedicated to a generalization of the notion of Newton polytope
(of a Laurent polynomial). We introduce {the notion of {\it Newton-Okounkov body} and prove a series of results
 about it}. It is a completely expanded and revised version of the second part of the preprint \cite{Askold-Kiumars-arXiv-1}. A revised and extended version of the first part can be found in \cite{Askold-Kiumars-arXiv-2}.
Nevertheless, the present paper is totally independent and self-contained.
Here we develop a geometric approach to semigroups in $\z^n$ and apply the results to graded
algebras, intersection theory and convex geometry.

A generalization of the notion of Newton polytope was started by the pioneering works of A. Okounkov
\cite{Okounkov-Brunn-Minkowski, Okounkov-log-concave}. A systematic study of the Newton-Okounkov body
was introduced about the same time in the papers \cite{Lazarsfeld-Mustata} and \cite{Askold-Kiumars-arXiv-1}.
Recently the Newton-Okounkov body (which Lazarsfeld-Mustata call Okounkov body) has been explored and
used in the papers \cite{Yuan}, \cite{Nystrom}, \cite{Jow} and \cite{BC}. {There is also the nice recent paper \cite{Kuronya} which
describes these bodies in some interesting cases}.

First, we briefly discuss the results we need from \cite{Askold-Kiumars-arXiv-2} and then
we will explain the results of the present paper in more details. For the sake of simplicity
throughout the introduction we may use slightly simplified notation compared to the rest of the paper.

The remarkable Bernstein-Kushnirenko theorem
computes the number of solutions of a system of equations  $P_1=\dots=P_n=0$ in $(\c^*)^n$, where
each $P_i$ is a {general} Laurent polynomial taken from a non-zero finite dimensional subspace $L_i$ spanned by
Laurent monomials. The answer is given in terms of the mixed
volumes of the Newton polytopes of the polynomials $P_i$. (The Kushnirenko
theorem deals with the case where the Newton polytopes of all the equations
are the same; the Bernstein theorem concerns the general case.)

In \cite{Askold-Kiumars-arXiv-2} a much more general situation is addressed.
Instead of $(\c^*)^n$ one takes any irreducible
$n$-dimensional algebraic variety  $X$, and instead of the finite dimensional
subspaces $L_i$ spanned by monomials one takes arbitrary non-zero finite dimensional subspaces
of rational functions on $X$. We denote the collection of all the non-zero finite dimensional subspaces of
rational functions on $X$ by $\K$.
For an $n$-tuple $L_1,\dots,L_n\in \K$, we define the intersection index
$[L_1,\dots,L_n]$ as the number of solutions in $X$ of a system of equations
$f_1=\dots=f_n=0$, where each $f_i$ is a {general} element in $L_i$. In counting the number of solutions
one neglects the solutions at which all the functions from a subspace $L_i$, for
some $i$, are equal to $0$, and the solutions at which at least one
function in $L_i$, for some $i$, has a pole. One shows that this intersection
index is well-defined and has all the properties of the intersection index of divisors
on a complete variety. There is a natural multiplication in the set $\K$. For $L,M\in \K$
the product $LM$ is the span of all the functions $fg$, where $f\in L$,
$g\in M$. With this product, the set $\K$ is a commutative semigroup. Moreover,
the intersection index is multi-additive with respect to this
product and hence can be extended
to the Grothendieck group of $\K$, which we denote by ${\bf G}_{rat}(X)$ (see Section \ref{subsec-int-index}).
If $X$ is a normal projective variety,
the group of (Cartier) divisors on $X$ can be embedded as a subgroup in the group ${\bf G}_{rat}(X)$.
Under this embedding, the intersection index in the group of divisors coincides
with the intersection index in the group ${\bf G}_{rat}(X)$. Thus the intersection index in
${\bf G}_{rat}(X)$ can be considered as a generalization of the classical intersection index of divisors,
which is birationally invariant and can be applied to non-complete varieties also
(as discussed in \cite{Askold-Kiumars-arXiv-2} all the properties of this generalized intersection index can
be deduced from the classical intersection theory of divisors).

Now about the contents of the present paper:
we begin with proving general (and not very hard) results regarding
a large class of semigroups of integral points.
The origin of our approach goes back
to \cite{Askold-Hilbert-poly}. 
Let us start with a class of semigroups with
a simple geometric construction: for an integer $0 \leq q <n$,
let $L$ be a $(q+1)$-dimensional rational subspace in $\r^n$, $C$
a $(q+1)$-dimensional {closed convex cone} in $L$ with apex at the origin, and $G$ a subgroup of full
rank $q+1$ in $L \cap \z^n$. \footnote{A linear subspace of $\r^n$ is called {\it rational} if it can be spanned
by rational vectors (equivalently integral vectors). An affine subspace is said to be {\it rational} if
it is a rational subspace after being shifted to pass through the origin.}
The set $\tilde{S} = G \cap C$ is a semigroup with respect to addition.
(After a linear change of coordinates, we can assume that the group $G$ coincides with $L \cap \z^n$
and hence $\tilde{S} = C \cap \z^n$.)
In addition, assume that the cone $C$ is strongly convex, that is, $C$ does not
contain any line. Let
$M_0 \subset L$ be a rational $q$-dimensional linear subspace which intersects $C$ only at the origin.
Consider the family of rational $q$-dimensional affine subspaces in $L$ parallel to $M_0$ such that
they intersect the cone $C$ as well as the lattice $G$.
Let $M_k$ denote the affine subspace in this family which has distance $k$ from the origin.
Let us normalize the distance $k$ so that as values it takes all the non-negative integers.
Then, this family of parallel affine subspaces can be enumerated as $M_0, M_1, M_2, \ldots$.
It is not hard to estimate the number $H_{\tilde{S}}(k)$ of points in the set
$\tilde{S}_k = M_k \cap \tilde{S}$. For sufficiently large $k$,
$H_{\tilde{S}}(k)$ is approximately equal to the
(normalized in the appropriate way) $q$-dimensional volume of the convex body $C \cap M_k$.
This idea, which goes back to Minkowski, shows that $H_{\tilde{S}}(k)$ grows like
$a_qk^q$ where the {\it $q$-th growth coefficient} $a_q$ is equal to the (normalized) $q$-dimensional volume of
the convex body $\Delta(\tilde{S}) = C \cap M_1$.\footnote{For a function $f$, we define
the {\it $q$-th growth coefficient} $a_q$ to be the limit $\lim_{k\to \infty} f(k)/k^q$
(whenever this limit exists).}

We should point out that the class of semigroups $\tilde{S}$ above has already
a rich and interesting geometry even when $C$ is just a simplicial cone. For example,
it is related to a higher dimensional generalization of continued fractions originating
in the work of V. I. Arnold \cite{Arnold}.

Now let us discuss the case of a general semigroup of integral points.
Let $S \subset \z^n$ be a semigroup. Let $G$ be the subgroup of $\z^n$ generated by $S$, $L$
the subspace of $\r^n$ spanned by $S$, and $C$ the closure of the convex hull of $S \cup \{0\}$, that is,
the smallest closed convex cone 
(with apex at the origin) containing $S$. Clearly, $G$ and $C$ are
contained in the subspace $L$.
We define the {\it regularization} $\tilde{S}$ of $S$ to be the semigroup $C \cap G$.
\footnote{There is also the closely related notion of
{\it saturation} of a semigroup $S$. The saturation of $S$ is the semigroup of all $x \in \z^n$ for which $kx \in S$ for some
positive integer $k$. When $S$ is finitely generated, the saturation is the semigroup $C \cap \z^n$. Note that even when
$S$ is finitely generated, the saturation of $S$ can be different from the regularization of $S$, as the group $G$ can be
strictly smaller than $\z^n$.}.
From the definition $\tilde{S}$ contains $S$. We prove
that {\it the regularization $\tilde{S}$ asymptotically approximates the semigroup $S$.}
We call this the {\it approximation theorem}. More precisely:

\begin{THM} Let $C' \subset C$ be a {closed strongly convex cone} which intersects
the boundary (in the topology of the linear space $L$) of the cone $C$
only at the origin. Then there exists a constant $N>0$ (depending on $C'$)
such that any point in the group $G$ which lies in
$C'$ and whose distance from the origin is bigger than $N$ belongs to $S$.
\end{THM}

Now, in addition, assume that the cone $C$ constructed from $S$ is strongly convex.
Let $\dim L = q+1$. Fix a rational $q$-dimensional subspace $M_0 \subset L$ intersecting $C$ only at the origin
and as above let $M_k$, $k \in \z_{\geq 0}$, be the family of $q$-dimensional affine subspaces parallel to $M_0$.
That is, each $M_k$ intersects the cone $C$ as well as the group $G$. Let $H_S(k)$ and $H_{\tilde{S}}(k)$ be the number of points in the levels $S_k = S \cap M_k$ and
$\tilde{S}_k = \tilde{S} \cap M_k$ respectively.
The function $H_S$ is called the {\it Hilbert function} of the
semigroup $S$.

Let $\Delta(S) = C \cap M_1$. {One observes that it is a convex body}.
\footnote{{A {\it convex body} is a convex compact subset of $\r^n$.}}
We call it the {\it Newton-Okounkov body of the semigroup $S$}.
Note that $\dim \Delta(S) = q$.
By the above discussion (Minkowski's observation)
the Hilbert function $H_{\tilde{S}}(k)$ grows like $a_q k^q$ where
$a_q$ is the (normalized) $q$-dimensional volume of $\Delta(S)$.
But, by the approximation theorem, the Hilbert functions
$H_S(k)$ and $H_{\tilde{S}}(k)$ have the same asymptotic, as $k$ goes to infinity.
It thus follows that the volume of $\Delta(S)$ is responsible for the
asymptotic of the Hilbert function $H_S$ as well, i.e.

\begin{THM} \label{intro-THM-Hilbert-semigp}
The function $H_S(k)$ grows like
$a_qk^q$ where $q$ is the dimension of the convex body $\Delta(S)$, and the {\it $q$-th growth coefficient
$a_q$} is equal to the (normalized in the appropriate way) $q$-dimensional volume of $\Delta(S)$.
\end{THM}

More generally, we extend the above theorem to the sum of values of a polynomial on the points in
the semigroup $S$ (Theorem \ref{th-5.2}).

Next, we describe another result about the asymptotic behavior of a semigroup $S$.
With each non-empty level $S_k = C \cap M_k$ we can associate a subsemigroup $\widehat{S}_k \subset S$
generated by this level. It is non-empty only at the levels $kt$, $t \in \n$.
Consider the Hilbert function $H_{\hat{S}_k}(kt)$ equal to the number of
points in the level $kt$, of the semigroup $\hat{S}_k$. Then if $k$ is sufficiently large,
$H_{\hat{S}_k}(kt)$, regarded as a function of $t \in \n$, grows like $a_{q,k}t^q$ where the $q$-th growth coefficient $a_{q,k}$ depends on $k$.
We show that:
\begin{THM} \label{intro-THM-Fujita-semigp}
The growth coefficient $a_{q,k}$ for the function $H_{\widehat{S}_k}$,
considered as a function of $k$, has the same asymptotic as the Hilbert function $H_S(k)$ of the original
semigroup $S$.
\end{THM}


Now we explain the results in the paper on graded algebras.
Let $F$ be a finitely generated field of transcendence degree $n$ over $\k = \c$.
\footnote{For simplicity, here in the introduction we
take the ground field $\k$ to be $\c$, although throughout the paper, most of the results
are stated for a general algebraically closed field $\k$.}
Let $F[t]$ be the algebra of polynomials over $F$. We will be concerned with the graded $\k$-subalgebras of $F[t]$
and their Hilbert functions. In order to apply the results about the semigroups to graded
subalgebras of $F[t]$ one needs a valuation $v_t$ on the algebra $F[t]$.
Let $I$ be an ordered abelian group. An $I$-valued valuation on an algebra $A$ is a
map from $A \setminus \{0\}$ to $I$ which respects the algebra operations
(see Section \ref{subsec-valuation} for the precise definition).
We construct a $\z^{n+1}$-valued valuation $v_t$ on $F[t]$
by extending a valuation $v$ on $F$.
We also require $v$ to be {\it faithful}, i.e. it takes all the values in $\z^n$.
It is well-known how to construct many such valuations $v$. We present main examples in Section \ref{subsec-valuation}.

The valuation $v_t$ maps the set of non-zero elements
of a graded subalgebra $A \subset F[t]$ to a semigroup of
integral points in $\z^n \times \z_{\geq 0}$. This gives a connection between the
graded subalgebras of $F[t]$ and semigroups in $\z^n \times \z_{\geq 0}$.

The following types of graded subalgebras in $F[t]$ will play the main roles for us:
\begin{itemize}
\item[-] The algebra $A_L = \bigoplus_{k \geq 0}L^kt^k$, where $L$ is a non-zero finite dimensional
subspace of $F$ over $\k$. Here $L^0=\k$ and for $k>0$ the space $L^k$ is the span of all the products $f_1\cdots f_k$ with $f_1, \ldots ,f_k \in L$. It is a graded algebra generated by $\k$ and finitely many degree
$1$ elements.

\item[-] An {\it algebra of integral type} is a graded subalgebra $A$ which is a finite module over some algebra $A_L$, equivalently, a graded subalgebra which is finitely generated and {a finite module over the subalgebra generate by $A_1$}.

\item[-] An {\it algebra of almost integral type} is a graded subalgebra $A$ which is contained in an algebra of integral type, equivalently, a graded subalgebra which is contained in some algebra $A_L$.
\end{itemize}

{Let $X$ be an $n$-dimensional irreducible variety over $\k$ with $F = \k(X)$ its field of rational functions.
To a subspace $L\in \K$
one associates {the Kodaira rational map} $\Phi_L: X \ratmap \p(L^*)$, where
$\p(L^*)$ is the projectivization of the dual space to $L$. Take a point $x \in X$  such that all the
$f \in L$ are defined at $x$ and not all are zero at $x$. To $x$ there
corresponds a functional $\xi_x$ on $L$ given by
$\xi_x(f) = f(x)$. The Kodaira map $\Phi_L$ sends $x$ to the image of this functional
in the projective space $\p(L^*)$. Let $Y_L \subset \p(L^*)$ be the closure of the image of
$X$ under the map $\Phi_L$. The algebra $A_L$, in fact, can be identified with
the homogeneous coordinate ring of $Y_L \subset \p(L^*)$.
Algebras of integral type are related to the rings of sections of ample line bundles, and
algebras of almost integral type to the rings of sections of arbitrary line bundles
(see Theorems \ref{th-ring-sec-almost-finite-type} and \ref{th-very-ample-int-closure}).}

By the Hilbert-Serre theorem on finitely generated modules over a polynomial ring, it follows that the
Hilbert function $H_A(k)$ of an algebra $A$ of almost integral type does not grow faster than $k^n$.
{From this one can then show that the cone $C$ associated to the semigroup $S(A) = v_t(A \setminus \{0\})$ is
strongly convex.} Let $\Delta(A)$ denote the Newton-Okounkov body of the semigroup $S(A)$.
We call $\Delta(A)$ the {\it Newton-Okounkov body of the algebra $A$}.
Applying Theorem \ref{intro-THM-Hilbert-semigp} above we prove:
\begin{THM} \label{intro-THM-Hilbert-alg}
1) After appropriate rescaling of the argument $k$, the Hilbert function
$H_A(k)$ grows like $a_qk^q$, where $q$ is an integer between $0$ and $n$.
2) Moreover, the degree $q$ is equal to the dimension of $\Delta(A)$, and
$a_q$ is the (normalized in the appropriate way) $q$-dimensional volume of
$\Delta(A)$.
\end{THM}

When $A$ is of integral type, again by the Hilbert-Serre theorem, the
Hilbert function becomes a polynomial of degree $q$ for large values of $k$ and
the number $q!a_q$ is an integer. When $A$ is of almost integral type, the Hilbert function $H_A$ is
not in general a polynomial for large $k$ and $a_q$  can be transcendental. It seems to the authors that
the result above on the polynomial growth of the Hilbert function of algebras of almost integral type is new.

The Fujita approximation theorem in the theory of divisors states that the so-called
{\it volume of a big divisor} can be approximated by the
self-intersection numbers of ample divisors (see \cite{Fujita}, \cite[Section 11.4]{Lazarsfeld}).
In this paper, we prove an abstract analogue of the Fujita approximation theorem for algebras
of almost integral type. This is done by reducing  it, via the valuation $v_t$,
to the corresponding result for the semigroups (Theorem \ref{intro-THM-Fujita-semigp} above).
With each non-empty homogeneous
component $A_k$ of the algebra $A$ one associates the graded subalgebra $\widehat
A_k$ generated by this component. For fixed large enough $k$,
the Hilbert function $H_{\widehat A_{k}}(kt)$ of the algebra  $\widehat
A_k$ grows like  $a_{q,k} t^q$.

\begin{THM} \label{intro-THM-Fujita-alg}
The $q$-th growth coefficient $a_{q,k}$
of the Hilbert function  $H_{\widehat A_k}$, regarded as a function of $k$,
has the same asymptotic as the Hilbert function $H_A(k)$ of the algebra $A$.
\end{THM}

Hilbert's theorem on the dimension and degree of a projective variety yields an
algebro-geometric interpretation of the above results.
{Consider the algebra $A_L$ associated to a subspace $L \in \K$, and let $Y_L$ denote the
closure of the image of the Kodaira map $\Phi_L$.
Then by Hilbert's theorem we see that: {\it the dimension of the variety $Y_L$ is equal
to the dimension $q$ of the body $\Delta(A_L)$,
and the degree of $Y_L$ (in the projective space $\p(L^*)$) is equal to $q!$ times the $q$-dimensional
(normalized in the appropriate way) volume of $\Delta(A_L)$}.}

One naturally defines a componentwise product of graded subalgebras
(see Definition \ref{def-componentwise}).
Consider the class of graded algebras of almost integral type such that, for large enough $k$,
all their $k$-th homogeneous
components are non-zero. Let $A_1$, $A_2$ be algebras of such kind and put
$A_3 = A_1A_2$. It is easy to verify the inclusion
$$\Delta_0(A_1)+\Delta_0(A_2)\subset \Delta_0(A_3),$$
where $\Delta_0(A_i)$ is the Newton-Okounkov body for the algebra $A_i$
projected to $\r^n$ (via the projection on the first factor $\r^n \times \r \to \r^n$).
Using the previous result on the $n$-th growth coefficient $a_n(A_i)$ of the
Hilbert function of the algebra $A_i$ and the classical Brunn-Minkowski inequality
we then obtain the following inequality:
\begin{THM}
$$ a_n^{1/n}(A_1)+ a_n^{1/n}(A_2)  \leq a_n^{1/n}(A_3).$$
\end{THM}

The results about graded subalgebras of polynomials in particular
apply to the ring of sections of a divisor. In Section \ref{subsec-linear-series} we
see that the ring of sections of a divisor
is an algebra of almost integral type. Applying the above results to this algebra
we recover several well-known results regarding the asymptotic theory of divisors and linear
series. Moreover, we obtain some new results about the case when the divisor is not a big divisor.
As a corollary of our Theorem \ref{intro-THM-Fujita-alg} we generalize the interesting
Fujita approximation result in \cite[Theorem 3.3]{Lazarsfeld-Mustata}.
The result in \cite{Lazarsfeld-Mustata} applies to the so-called {\it big divisors} (or more generally
big graded linear series) on a projective variety.
Our generalization holds for any divisor (more generally any graded linear series)
on any complete variety (Corollary \ref{cor-linear-series}).
The point is that beside following the ideas in \cite{Lazarsfeld-Mustata},
we use results which apply to {\it arbitrary} semigroups of integral points. Another difference
between the approach in the present paper and that of \cite{Lazarsfeld-Mustata}
is that we use abstract valuations on algebras, as opposed to
a valuation on the ring of sections of a line bundle and coming from a flag of subvarieties.
On the other hand, the use of special valuations with algebro-geometric nature is helpful to get
more concrete information about the Newton-Okounkov bodies in special cases.

Let us now return to the subspaces of rational functions on a variety $X$.
Let $L \in \K$. If the Kodaira map $\Phi_L:X \ratmap \p(L^*)$ is a birational isomorphism between
$X$ and its image $Y_L$ then the degree of $Y_L$
is equal to the self-intersection index $[L,\dots,L]$
of the subspace $L$. We can then apply the results above to
the intersection theory on $\K$. Let us call a subspace $L$
a {\it big subspace} if for large $k$, $\Phi_{L^k}$ is a birational
isomorphism between $X$ and $Y_{L^k}$.

With a space $L\in \K$, we associate two graded algebras:
the algebra $A_L$ and its integral closure
$\ol{A_L}$ {in the field of fractions of the polynomial algebra $F[t]$.}
The algebra $A_L$ is easier to define and fits our purposes
best when the subspace $L$ is big. On the other hand, the second algebra
$\ol{A_L}$ is a little bit more complicated to define (it involves the integral closure)
but leads to more convenient results for any $L \in \K$ (Theorem \ref{intro-THM-Kus} below).
The algebraic construction of going from $A_L$ to its integral closure $\ol{A_L}$ can
be considered as the analogue of the geometric operation of taking the convex hull of a set of
points.

One can then associate to $L$ two
convex bodies $\Delta(A_L)$ and $\Delta(\ol{A_L})$. In general
$\Delta(A_L) \subseteq \Delta(\overline{A_L})$, while for a big subspace $L$
we have $\Delta(A_L)=\Delta(\ol{A_L})$.

The following generalization of the
Kushnirenko theorem gives a geometric interpretation of the self-intersection
index of a subspace $L$:
\begin{THM} \label{intro-THM-Kus}
For any $n$-dimensional irreducible algebraic variety  $X$
and for any $L\in \K$ we have:
\begin{equation} \label{equ-*}
[L, \dots, L] = n!\Vol(\Delta(\ol{A_L})).
\end{equation}
\end{THM}

The Kushnirenko theorem is a special case of the formula (\ref{equ-*}).
The Newton polytope of the product of two Laurent polynomials is equal to
the sum of the corresponding Newton polytopes. This additivity property  of
the Newton polytope and multi-additivity of the intersection index in
$\K$ give the Bernstein theorem as a corollary of the Kushnirenko theorem.

{Both} of the bodies $\Delta(A_L)$ and
$\Delta(\overline{A_L})$ satisfy superadditivity property, that is,
{\it the convex body associated to the product of two subspaces, contains the sum of
the convex bodies corresponding to the subspaces.}

The formula (\ref{equ-*}) and the superadditivity
of the Newton-Okounkov body $\Delta(\ol{A_L})$ together with the
classical Brunn-Minkowski inequality for convex bodies, then imply an analogous
inequality for the self-intersection index:
\begin{THM}
Let $L_1, L_2 \in \K$ and put $L_3=L_1L_2$. We have:
$$[L_1,\dots,L_1]^{1/n} + [L_2,\dots,L_2]^{1/n} \leq [L_3,\dots, L_3]^{1/n}.$$
\end{THM}

For an algebraic surface $X$, i.e. for $n=2$, this inequality is equivalent
to the following analogue of the Hodge inequality (from the Hodge index
theorem):
\begin{equation} \label{equ-**}
[L_1,L_1] [L_2,L_2] \leq [L_1,L_2]^2.
\end{equation}

The Hodge index theorem holds for smooth irreducible projective
(or compact Kaehler) surfaces. Our inequality (\ref{equ-**}) holds for any
irreducible surface, not necessarily smooth or complete, and hence is easier to apply.
In contrast to the usual proofs of the Hodge inequality,
our proof of the inequality (\ref{equ-**}) is completely elementary.

Using properties of the intersection index in $\K$
and using the inequality (\ref{equ-**})
{one can easily prove the algebraic analogue of
Alexandrov-Fenchel inequality (see Theorem \ref{th-Alex-Fenchel-alg}).
The classical Alexandrov-Fenchel inequality (and its many corollaries) in convex geometry follow
easily from its algebraic analogue via the Bernstein-Kushnirenko theorem}.
These inequalities from intersection theory
and their application to deduce the corresponding inequalities
in convex geometry, have been known (see \cite{Askold-BZ}, \cite{Teissier}).
A contribution of the present paper is an elementary proof
of the key inequality (\ref{equ-**}) which makes all the chain of arguments involved elementary
and more natural.

This paper stems from an attempt to understand the right definition of the Newton polytope for
actions of reductive groups on varieties. Unexpectedly, we found that one can define many convex bodies (i.e.
Newton-Okounkov bodies) analogous to the Newton polytope and their definition, in general,
is not related with the group action.
It is unlikely that one can completely understand the shape of a Newton-Okounkov body in
the general situation (see \cite{Kuronya} for some results in this direction).
{In \cite{Askold-Kiumars-reductive}, we return to reductive group actions and
consider the Newton-Okounkov bodies associated to invariant subspaces of rational functions on spherical varieties
and constructed via special valuations (see also \cite{Askold-Kiumars-Kaz} and \cite{Askold-Kiumars-horo}).}
The Newton-Okounkov bodies in such cases can be described (in particular they are convex polytopes)
and the results of the present paper become more concrete.\\

\noindent{\bf Acknowledgement:} We would like to thank the referee for careful reading of the
manuscript and giving numerous helpful suggestions.

\section{Part I: Semigroups of integral points}
In this part we develop a geometric approach to the semigroups of integral
points in $\r^n$. The origin of this approach goes back to the paper \cite{Askold-Hilbert-poly}.
We show that a semigroup of integral points is sufficiently close to the semigroup of all points
in a sublattice and lying in a convex cone in $\r^n$. We then introduce the notion of Newton-Okounkov body
for a semigroup, which is responsible for the asymptotic
of the number of points of the semigroup in a given (co)direction. Finally, we prove a theorem which compares the
asymptotic of a semigroup and that of its subsemigroups. We regard this as an abstract version of
the Fujita approximation theorem in the theory of divisors. Later in the paper, the results of this part
will be applied to graded algebras and to intersection theory.

\subsection{Semigroups of integral points and their regularizations} \label{subsec-semigp-regularization}
Let $S$ be an additive semigroup in the lattice $\z^n \subset \r^n$.
In this section we will define the {\it regularization of $S$}, a simpler
semigroup with more points constructed out of the semigroup $S$.
The main result is the approximation theorem (Theorem \ref{th-semi-gp-approximation})
which states that the regularization of $S$ asymptotically approximates $S$.
Exact definitions and statement will be given below.

To a semigroup $S$ we associate the following basic objects:
\begin{Def}
\begin{enumerate}
\item The {\it subspace generated by the semigroup $S$} is the real span
$L(S) \subset \r^n$ of the semigroup $S$.
By definition the linear space $L(S)$ is spanned by integral vectors and thus
{\it the rank of the lattice $L(S) \cap \z^n$ is equal to $\dim L(S)$}.

\item The {\it cone generated by the semigroup $S$} is the closed convex cone $\Con(S) \subset L(S)$
which is the closure of the set of all linear combinations $\sum_i \lambda_ia_i$ for
$a_i \in S$ and $\lambda_i \geq 0$.

\item The {\it group generated by the semigroup $S$} is the group $G(S) \subset L(S)$ generated by all the
elements in the semigroup $S$. The group $G(S)$ consists of all the linear combinations
$\sum_i k_ia_i$ where $a_i \in S$ and $k_i \in \z$.
\end{enumerate}
\end{Def}

\begin{Def}
The {\it regularization of a semigroup $S$} is the semigroup
$$\Reg(S)= G (S)\cap \Con(S).$$ Clearly the semigroup $S$ is contained in its
regularization.
\end{Def}


The {\it ridge} of a closed convex cone with apex at the origin is the biggest linear
subspace contained in the cone. A cone is called {\it strictly convex} if its ridge
contains only the origin. The {\it ridge} $L_0(S)$ of a semigroup $S$ is the ridge of the
cone $\Con(S)$.

First we consider the case of finitely generated semigroups.
The following statement is obvious:
\begin{Prop}
Let $A \subset \z^n$ be a finite set generating a semigroup $S$, and let
$\Delta(A)$ be the convex hull of $A$. Then:
1) The space $L(S)$ is the smallest subspace containing the polytope $\Delta(A)$.
2) The cone $\Con(S)$ is the cone with the apex at the origin over the polytope $\Delta(A)$.
3) If the origin $O$ belongs to $\Delta(A)$ then the ridge
$L_0(S)$ is the space generated by the smallest
face of the polytope $\Delta(A)$ containing $O$, otherwise $L_0(S) = \{O\}$.
\end{Prop}

{The following statement is well-known in toric geometry (conductor ideal). For the sake of completeness
we give a proof here.}
\begin{Th} \label{th-main-finite-gen}
Let $S \subset \z^n$ be a finitely generated semigroup. Then there is an element $g_0 \in S$
such that $\Reg(S) + g_0 \subset S$, i.e. for any element $g \in \Reg(S)$ we have
$g + g_0 \in S$.
\end{Th}
\begin{proof}
Let $A$ be a finite set generating $S$ and
let $P\subset \r^n$ be the set of vectors $x$ which can be represented in the form
$x=\sum \lambda_i a_i$, where $0\leq \lambda_i<1$ and $a_i\in A$.
The set $P$ is bounded and hence $Q=P \cap G(S)$ is finite.
For each $q\in Q$ fix a representation of $q$ in the form $q=\sum
k_i(q)a_i$, where $k_i(q)\in \z$ and $a_i\in A$. Let $g_0
=\sum_{a_i\in A} m_i a_i$, with $m_i= 1 -\min_{q\in Q} \{k_i(q)\}$.
Each vector $g\in \Reg(S)\subset \Con(S)$ can be represented in the form
$g=\sum \lambda_ia_i$, where$\lambda_i\geq 0$ and $a_i\in A$.  Let
$g=x+y$, with $x= \sum [\lambda_i] a_i$ and $y=\sum
(\lambda_i-[\lambda_i])a_i$. Clearly $x\in S \cup \{0\}$ and $y\in
P$. Let's verify that $g+g_0\in S$. In fact $g+g_0=x +
(y+g_0)$. Because $g\in \Reg(S)$, we have $y\in Q$. Now $y+g_0=\sum
k_i(y)a_i+\sum m_ia_i= \sum (k_i(y)+m_i)a_i$. By definition
$k_i(y)+m_i\geq 1$ and so $(y+g_0)\in S$. Thus $g+g_0=x +(y+g_0)\in S$.
This finishes the proof.
\end{proof}


Fix any Euclidean metric in $L(S)$.
\begin{Cor}
Under the assumptions of Theorem \ref{th-main-finite-gen}, there is a constant $N>0$
such that any point in $G(S) \cap \Con(S)$ whose distance to the boundary of $\Con(S)$ (as a subset of
the topological space $L(S)$) is bigger than or equal to
$N$, is in $S$.
\end{Cor}
\begin{proof}
It is enough to take $N$ to be the length of the vector $g_0$ from Theorem \ref{th-main-finite-gen}.
\end{proof}

Now we consider the case where the semigroup $S$ is not necessarily finitely
generated. Let $S \subset \z^n$ be a semigroup and let $\Con$ be {a closed strongly convex cone}
inside $\Con(S)$ which intersects the boundary of $\Con(S)$ (as a subset of $L(S)$)
only at the origin. We then have:

\begin{Th}[Approximation of a semigroup by its regularization] \label{th-semi-gp-approximation}
There is a constant $N>0$ (depending on the choice of $\Con \subset \Con(S)$)
such that each point in the group $G(S)$
which lies in $\Con$ and whose distance from the origin is bigger than
$N$ belongs to $S$.
\end{Th}
{We will need a simple lemma:}
\begin{Lem} \label{lem-in-approx-th}
Let $C' \subset C \subset \r^n$ be closed convex cones with apex at the origin. Moreover,
assume that the boundaries of $C$ and $C'$ (in the
topologies of their linear spans) intersect only at the origin. Take $x_0 \in \r^n$. Then the shifted cone $x_0 +C'$
contains all the points in $C$ which are far enough from the origin.
\end{Lem}
\begin{proof}
{
Consider $B' = \{x \in C' \mid |x| = 1\}$. Then, as the boundaries of $C$ and $C'$ intersect only at the origin,
$B'$ is a compact subset of $C$ which lies in the
interior of $C$. Thus there exist $R > 0$ such that for any $r > R$ we have $(x_0/r) + B' \subset C$. But since $C$ is a cone
we conclude that $x_0 + rB' \subset C$ which proves the claim.
}
\end{proof}

\begin{proof}[Proof of Theorem \ref{th-semi-gp-approximation}]
Fix a Euclidean metric in $L(S)$ and equip $L(S)$ with the corresponding topology.
We will only deal with $L(S)$ and the ambient space $\r^n$ will not be used in the proof below.
Let us enumerate the points in the semigroup $S$ and let $A_i$ be the collection of the first $i$ elements of
$S$. Denote by $S_i$ the semigroup generated by $A_i$.
There is $i_0>0$ such that for $i > i_0$ the set $A_i$ contains a set of generators for the group $G(S)$.
If $i>i_0$ then the group $G(S_i)$ generated by the semigroup $S_i$ coincides with $G(S)$, and the
space $L(S_i)$ coincides with $L(S)$.

Fix any linear function $\ell: L(S) \to \r$ which is strictly positive
on $\Con \setminus \{0\}$. Let $\Delta_\ell(\Con(S))$ and $\Delta_\ell(\Con)$
be the closed convex sets obtained by intersecting $\Con(S)$ and
$\Con$ by the hyperplane $\ell=1$ respectively. By definition $\Delta_\ell(\Con)$ is bounded
and is strictly inside $\Delta_\ell(\Con(S))$.

The convex sets $\Delta_\ell(\Con(S_i))$, obtained by intersecting
$\Con(S_i)$ with the hyperplane $\ell=1$, form an increasing sequence of
closed convex sets in this hyperplane. The closure of the union of the
sets $\Delta_\ell(\Con(S_i))$ is, by construction, the convex set $\Delta_\ell(\Con(S))$.
So there is an integer $i_1$ such that for $i>i_1$ the set
$\Delta_\ell(\Con)$ is strictly inside $\Delta_\ell(\Con(S_i))$.
Take any integer $j$ bigger than $i_0$ and $i_1$. By Theorem \ref{th-main-finite-gen}, for
the finitely generated semigroup $S_j$ there is a vector $g_0$
such that any point in $G(S)\cap (g_0+ \Con(S_j))$ belongs to $S$.
The convex cone $\Con$ is contained in $\Con(S_j)$ and their boundaries intersect only
at the origin. {Now by Lemma \ref{lem-in-approx-th}
the shifted cone $g_0 + \Con(S_j)$ contains all the points of
$\Con$ which are far enough from the origin.}
This finishes the proof of the theorem.
\end{proof}

\begin{Ex} \label{ex-concave-graph}
In $\r^2$ with coordinates $x$ and $y$, consider the domain $U$ defined by the inequality
$y \geq F(x)$ where $F$ is an even function, i.e. $F(x)=F(-x)$, such that $F(0)=0$
and $F$ is concave and increasing on the ray $x \geq 0$. The set $S = U \cap \z^2$ is a
semigroup. The group $G(S)$ associated to this semigroup is $\z^2$. The cone $\Con(S)$
is given by the inequality $y \geq c|x|$ where $c = \lim_{x \to \infty} F(x)/x$ and the
regularization $\Reg(S)$ is $\Con(S) \cap \z^2$. In particular, if $F(x)= |x|^{\alpha}$ where
$0 < \alpha < 1$, then $\Con(S)$ is the half-plane $y \geq 0$ and $\Reg(S)$ is the set of integral points
in this half-plane. Here the distance from the point $(x,0) \in \Con(S)$ to the semigroup
$S$ goes to infinity as $x$ goes to infinity.
\end{Ex}


\subsection{Rational half-spaces and admissible pairs}
In this section we discuss admissible pairs consisting of a semigroup and a half-space.
We define the Newton-Okounkov body and the Hilbert function for an admissible pair.

Let $L$ be a linear subspace in $\r^n$ and $M$ a half-space in $L$ with boundary $\partial M$.
A half-space $M \subset L$ is {\it rational} if the subspaces $L$ and $\partial M$ {can be spanned by integral
vectors, i.e. are rational subspaces}.

With a rational half-space $M \subset L$ one can associate
$\partial M_\z= \partial M \cap \z^n$ and $L_\z =L \cap \z^n$.
Take the linear map $\pi_M: L \to \r$ such that $\ker(\pi_M) = \partial M$, $\pi_M(L_\z) = \z$ and
$\pi_M(M \cap \z^n) = \z_{\geq 0}$, the set of all non-negative integers.
The linear map $\pi_M$ induces an isomorphism from $L_\z/\partial M_\z$ to $\z$.

Now we define an admissible pair of a semigroup and
a half-space.
\begin{Def}
A pair $(S, M)$ where $S$ is a semigroup in $\z^n$ and $M$ a rational
half-space in $L(S)$ is called {\it admissible} if $S \subset M$. We call
an admissible pair $(S, M)$ {\it strongly admissible} if the cone
$\Con(S)$ is strictly convex and intersects the space $\partial M$ only at the
origin.
\end{Def}

With an admissible pair $(S, M)$ we associate the following objects:
\begin{itemize}
\item[-] $\ind(S,\partial M)$, the index of the subgroup $G(S)\cap \partial M$
in the group $\partial M_{\z}$.

\item[-] $\ind(S,M)$, the index of the subgroup $\pi_M (G(S))$ in the group $\z$.
(We will usually denote $\ind(S,M)$ by the letter $m$.)

\item[-] $S_k$, the subset $S\cap \pi^{-1}_M(k)$ of the points of $S$ at level $k$.
\end{itemize}

\begin{Def}
The {\it Newton-Okounkov convex set} $\Delta(S, M)$ of an admissible pair
$(S,M)$, is the convex set $\Delta(S,M)= \Con(S)\cap \pi^{-1}_M(m)$,
where $m=\ind(S,M)$. It follows from the definition that the convex set $\Delta(S, M)$ is compact
({i.e. is a convex body})
if and only if the pair $(S, M)$ is strongly admissible.
In this case we call $\Delta(S, M)$ the {\it Newton-Okounkov body} of $(S, M)$.
\end{Def}

We now define the Hilbert function of an admissible pair $(S, M)$. It is convenient to define it
in the following general situation. Let $T$ be a
commutative semigroup and $\pi: T \to \z_{\geq 0}$ a homomorphism of semigroups.

\begin{Def} 1) The {\it Hilbert function} $H$ of $(T, \pi)$ is
the function $H: \z_{\geq 0} \to \z_{\geq 0} \cup \{\infty\}$,
defined by $H(k) = \# \pi^{-1}(k)$. The {\it support} $\supp(H)$ of the Hilbert function
is the set of $k \in \z_{\geq 0}$ at which $H(k)\neq 0$.
2) The Hilbert function of an admissible pair $(S, M)$ is the Hilbert function of
the semigroup $S$ and the homomorphism $\pi_M: S \to \z_{\geq 0}$. That is,
$H(k) = \#S_k$, for any $k \in \z_{\geq 0}$.
\end{Def}

The following is easy to verify.
\begin{Prop} \label{prop-4.1} 
Let $T$ and $\pi$ be as above.
1) The support $\supp(H)$ of the Hilbert function $H$
is a semigroup in $\z_{\geq 0}$.
2) If the semigroup $T$ has the cancellation property then the set
$H^{-1}(\infty)$ is an ideal in the semigroup $\supp(H)$,
i.e. if $x\in H^{-1}(\infty)$ and $y\in \supp(H)$ then $x+y \in H^{-1}(\infty)$.
3) Let $m$ be the index of the subgroup generated by $\supp(H) \subset \z_{\geq 0}$ in $\z$.
Then $\supp(H)$ is contained in $m\z$ and there is a constant $N_1$ such that for $mk
> N_1$ we have $mk \in \supp(H)$.
4) If the semigroup $T$ has the cancellation property and $H^{-1}(\infty) \neq \emptyset$, then there is $N_2$ such that
for $mk>N_2$ we have $H(mk)=\infty$.
\end{Prop}
\begin{proof}
1) and 2) are obvious. 3) Follows from Theorem \ref{th-semi-gp-approximation} applied to the semigroup
$\supp(H) \subset \z$. Finally 4) Follows from 2) and 3). 
\end{proof}

In particular, if the Hilbert function of an admissible pair $(S, M)$ is equal to infinity for at least one $k$,
then for sufficiently large values of $k$, Proposition \ref{prop-4.1}(4) describes this function
completely. Thus in what follows we will assume that the Hilbert function always takes finite values.

\subsection{Hilbert function and volume of the Newton-Okounkov convex set}
In this section we establish a connection between the asymptotic of
the Hilbert function of an admissible pair and its Newton-Okounkov body.

First let us define the notion of integral volume in a rational affine subspace.

\begin{Def}[Integral volume] \label{def-int-volume} Let $L \subset \r^n$ be a rational
linear subspace of dimension $q$. The {\it integral measure} in $L$ is the translation
invariant Euclidean measure in $L$ normalized such that the smallest measure of a $q$-dimensional
parallelepiped with vertices in $L \cap \z^n$ is equal to $1$.
Let $E$ be a rational affine subspace of dimension $q$ and parallel to $L$.
The {\it integral measure} on $E$ is the integral measure on $L$ shifted to $E$.
The measure of a subset $\Delta \subset E$ will be called its {\it integral volume} and denoted by $\Vol_q(\Delta)$.
\end{Def}

For the rest of the paper, unless otherwise stated, $\Vol_q$ refers to the integral volume.

Now let $(S, M)$ be an admissible pair with $m = \ind(S, M)$. Put $q = \dim \partial M$.
We denote the integral measure in the affine space $\pi_M^{-1}(m)$ by $d\mu$.
Take a polynomial $f:\r^n \to \r $ of degree $d$ and let $f=f^{(0)}+f^{(1)}+\dots+f^{(d)}$ be its decomposition into homogeneous components.

\begin{Th} \label{th-5.2} Let $(S, M)$ be a strongly admissible pair. Then
$$\lim_{k \to \infty}\frac {\sum _{x\in S_{mk}}f(x) }{k^{q+d}} =\frac{\int_{\Delta(S,M)}f^{(d)}(x)d\mu}{\ind  (S,\partial M)}.$$ \end{Th}

Let $M$ be the positive half-space $x_{q+1} \geq 0$ in $\r^{q+1}$. Take a $(q+1)$-dimensional
{closed strongly convex cone} $C \subset M$ which intersects $\partial M$ only at the origin. Let $S = C \cap \z^{q+1}$
be the semigroup of all the integral points in $C$. Then $(S, M)$ is a strongly admissible pair.
For such kind of a saturated semigroup $S$, Theorem \ref{th-5.2} is relatively easy to show. We restate the above theorem in this case as it will be needed in the proof of the general case.
Results of such kind have origins in the classical work of Minkowski.

\begin{Th} \label{th-sum-integral} Let $S = C \cap \z^{q+1}$ and
$\Delta = C \cap \{x_{q+1} = 1 \}$. Then:
$$\lim_{k \to \infty} \frac{\sum _{x \in S_k}f(x)}{k^{q+d}} =\int_{\Delta}f^{(d)}(x)d\mu.$$
Here $S_k$ is the set of all the integral points in $C \cap \{x_{q+1} = k\}$,
and $d\mu$ is the Euclidean measure at the hyperplane $x_{q+1}=1$.
\end{Th}
{Theorem \ref{th-sum-integral} can be easily proved by considering the Riemann sums for the integrals of the homogeneous components
of $f$ over $\Delta$}.

\begin{proof}[Proof of Theorem \ref{th-5.2}] The theorem follows from Theorem \ref{th-semi-gp-approximation}
(approximation theorem) and Theorem \ref{th-sum-integral}. Firstly, one reduces to the case where $L(S)= \r^{n}$, $q+1=n$, $M$ is given by the inequality $x_{q+1}\geq 0$, $G(S)= \z^{q+1}$, $\ind(S,\partial M)=\ind (S,M)=1$ and $\Delta(S,M)$ is a $q$-dimensional convex body in the hyperplane $x_{q+1}=1$, as follows: choose a basis $e_1,\dots,e_q, e_{q+1},\dots e_n$ in $\r^n$ such that $e_1,\dots,e_q$ generate the group $G(S)\cap \partial M$ and the vectors $e_1,\dots,e_{q+1}$ generate the group $G(S)$ (no condition on the rest of vectors in the basis). This choice of basis identifies the spaces $L(S)$ and $\partial M$ with $\r^{q+1}$ and $\r^q$ respectively. We will not deal with the vectors outside $\r^{q+1}$, and hence we can assume $q+1=n$. Under such choice of a basis the lattice $L(S)_\z$ identifies with a lattice $\Lambda \subset \r^{q+1}$ which may contain non-integral points. Also the lattice $\partial M_{\z}$ identifies with a lattice $\Lambda \cap \r^q$. The index of the subgroup $\z^q$ in the group $\Lambda \cap \r^q$ is equal to $\ind(S,\partial M)$. The coordinate $x_{q+1}$ of the points in the lattice $\Lambda\subset \z^{q+1}$ is proportional to the number $1/m$ where $m= \ind (S,M)$. The map $\pi_M: L(S)_\z/ M_{\z}\to \z$ then coincides with the restriction of the map $mx_{q+1}$ to the lattice $\Lambda$. The semigroup $S$ becomes a subsemigroup in the lattice $\z^q$ and the level set $S_k$ is equal to $S\cap \{x_{q+1}=k\}$. Also the measure $d\mu$ is given by $d\mu=\rho d{\bf x}=\rho dx_1\wedge\dots\wedge x_q$, where $\rho=\ind(S,\partial M)$. Thus with the above choice of basis the theorem is reduced to this particular case.

To prove that the limit exists and is equal to $\int_{\Delta(S,M)}f^{(d)}(x)d{\bf x}$, it is enough to show that any limit point of the sequence $\{g_k\}$, $g_k = \sum _{x\in S_{k}}f(x)/k^{q+d}$, lies in arbitrarily small neighborhoods of $\int_{\Delta(S,M)}f^{(d)}(x)d{\bf x}.$ Take a convex body $\Delta$ in the hyperplane $x_{q+1}=1$ which lies strictly inside the Newton-Okounkov body $\Delta(S,M)$. Consider the convex bodies $k \Delta (S,M)$ and $k \Delta $ in the hyperplane $x_{q+1}=k$. Let $S_k'$ and $S_k''$ be the sets $k\Delta\cap \z^{q+1}$ and $k\Delta(S,M)\cap \z^{q+1}$ respectively. By Theorem \ref{th-semi-gp-approximation}, for large values of $k$, we have $S_k'\subset S_k\subset S_k''$. Also by Theorem \ref{th-sum-integral}:
$$ \lim_{k\to \infty}\frac {\sum _{x\in S_k' }f(x) }{k^{q+d}} =\int_{\Delta}f^{(d)}(x)d{\bf x},$$
$$ \lim_{k\to \infty}\frac {\sum _{x\in S_k'' }f(x) }{k^{q+d}} =\int_{\Delta(S,M)}f^{(d)}(x)d{\bf x},$$
$$ \lim_{k\to \infty}\frac {\# (S_k''\setminus S_k')}{k^q} = \Vol_q(\Delta(S,M)\setminus \Delta).$$
Since $(S, M)$ is strongly admissible,
one can find a constant $N > 0$ such that for any point $x \in \Con(S)$ with $x_{q+1} \geq 1$ we have
$|f(x)|/x_{q+1}^d < N$ and $|f^{(d)}(x)|/x_{q+1}^d < N$. This implies that for large values of $k$ we have:
$$ \frac {\sum _{x\in (S_k''\setminus S_k') }|f(x)| }{k^{q+d}} \leq \tilde N \Vol_q(\Delta(S,M)\setminus \Delta),$$ $$\int_{\Delta(S,M)\setminus \Delta}|f^{(d)}(x)|d{\bf x}< \tilde N \Vol_q(\Delta(S,M)\setminus \Delta),$$
where $\tilde N$ is any constant bigger than $N$. Thus
$$ |\frac {\sum _{x\in S_k}f(x) }{k^{q+d}}
-\int_{\Delta(S,M)}f^{(d)}(x)d{\bf x}|< 2\tilde N \Vol_q(\Delta(S,M)\setminus \Delta).$$
For any given $\varepsilon > 0$ we may choose the convex body $\Delta$ such that $\Vol_q(\Delta(S,M)\setminus \Delta) < \varepsilon/2\tilde N$.
This shows that for any $\varepsilon>0$, all the limit points of the sequence $\{g_k\}$ belong to the $\varepsilon$-neighborhood of the number $\int_{\Delta(S,M)}f^{(d)}(x)d{\bf x}$, which finishes the proof.
\end{proof}

\begin{Cor} \label{cor-1.23}
With the assumptions as in Theorem \ref{th-5.2}, the following holds:
$$ \lim_{k\to \infty}\frac {\#
S_{mk}}{k^q} =\frac{\Vol_q (\Delta(S,M))}{\ind  (S,\partial M)}.$$
\end{Cor}
\begin{proof}
Apply Theorem \ref{th-5.2} to the polynomial $f=1$.
\end{proof}

\begin{Def}
Let $(S, M)$ be an admissible pair with $m = \ind(S, M)$ and $q = \dim \partial M$.
We say that $S$ has {\it bounded growth} with respect to
the half-space $M$ if there exists a sequence $k_i\to \infty$ of positive integers such that
the sets $S_{m k_i}$ are finite and the sequence of numbers $\#S_{mk_i}/ k_i^q$ is bounded.
\end{Def}

\begin{Th} \label{th-5.4}
Let $(S, M)$ be an admissible pair.
The semigroup $S$ has bounded growth with respect to $M$ if and only if
the pair $(S, M)$ is strongly admissible. In fact, if $(S, M)$ is strongly
admissible then $S$ has polynomial growth.
\end{Th}
\begin{proof}
Let us show that if $S$ has bounded growth then $(S, M)$ is strongly admissible.
Suppose the statement is false. Then the Newton-Okounkov convex set $\Delta(S, M)$ is
an unbounded convex $q$-dimensional set and hence has infinite $q$-dimensional volume.
Assume that $P$ is a constant such that for any $i$, $\#S_{mk_i}/ k_i^q < P$.
Choose a convex body $\Delta$ strictly inside $\Delta(S, M)$ in such a way that the $q$-dimensional
volume of $\Delta$ is bigger than $mP$. Let $\Con$ be the cone over the convex body
$\Delta$ with the apex at the origin. By Theorem \ref{th-semi-gp-approximation}, for large values of
$k_i$, the set $S_{mk_i}$ contains the set
$S_{m k_i}'=\Con \cap G(S)\cap \pi^{-1}_M (mk_i)$.
Then by Corollary \ref{cor-1.23},
$$\lim_{k_i\to \infty}\frac { \# S_{m k_i}'}{ k_i^q}
=\frac {\Vol_q(\Delta)}{\ind (S,\partial M)}>P.$$
The contradiction proves the claim. The other direction, namely if $(S, M)$ is
strongly admissible then it has polynomial growth (and hence bounded growth),
follows immediately from Corollary \ref{cor-1.23}.
\end{proof}


\begin{Th} \label{th-5.1}
Let $(S, M)$ be an admissible pair and assume that the sets $S_k$, $k\in \z_{\geq0}$, are finite. Let $H$ be the
Hilbert function of $(S, M)$ and put $\dim
\partial M=q$. Then
\begin{enumerate}
\item The limit
$$\lim_{k\to \infty} \frac{H(mk)}{k^q},$$ exists (possibly infinite),
where $m=\ind (S,M)$.
\item This limit is equal to the volume (possibly infinite) of the Newton-Okounkov convex set
$\Delta(S, M)$ divided by the integer $\ind(S,\partial M)$.
\end{enumerate}
\end{Th}
\begin{proof}
First assume that $H(mk)/k^q$ does not approach infinity (as $k$ goes to infinity).
Then there is a sequence $k_i\to \infty$ with $k_i\in \z_{\geq 0}$ such that the sets $S_{m k_i}$
are finite and the sequence $\#S_{mk_i}/ k_i^q$ is bounded. But this means that the semigroup
$S$ has bounded growth with respect to the half-space $M$.
Thus by Theorem \ref{th-5.4} the cone $\Con(S)$ is strictly convex and
intersects $\partial M$ only at the origin. In this case the theorem
follows from Corollary \ref{cor-1.23}. Now if $\lim_{k\to
\infty} H(mk)/k^q=\infty$, then the conditions in Theorem \ref{th-5.2} cannot be
satisfied. Hence the convex set $\Delta(S,M)$ is unbounded and thus has infinite volume. This
shows that Theorem \ref{th-5.1} is true in this case as well.
\end{proof}

\begin{Ex}
Let $S$ be the semigroup in Example \ref{ex-concave-graph} where $F(x) = |x|^{1/n}$ for some
natural number $n>1$. Also let $M$ be the half-space $y \geq 0$. Then the pair $(S, M)$ is
admissible. Its Newton-Okounkov set $\Delta(S, M)$ is the line $y=1$ and its Hilbert function
is given by $H(k) = 2k^n+1$. Thus in spite of the fact that the dimension of the Newton-Okounkov convex set
$\Delta(S, M)$ is $1$, the Hilbert function grows like $k^n$. This effect is related to the fact that the
pair $(S, M)$ is not strongly admissible.
\end{Ex}

\subsection{Non-negative semigroups and approximation theorem} \label{sec-non-negative}
In $\r^{n+1} = \r^n \times \r$ there is a natural half-space $\r^n \times \r_{\geq 0}$, consisting of
the points whose last coordinate is non-negative.
In this section we will deal with semigroups that are contained in this fixed half-space of
full dimension. For such semigroups we refine the statements of theorems proved in the previous sections.

We start with definitions.
A {\it non-negative} semigroup of integral points in
$\r^{n+1}$ is a semigroup $S \subset \r^n \times \r_{\geq 0}$ which is not contained in the
hyperplane $x_{n+1}=0$. With a non-negative semigroup $S$ we can associate an admissible pair $(S, M(S))$ where
$M(S) = L(S) \cap (\r^n \times \r_{\geq 0})$. We call a non-negative semigroup,
{\it strongly non-negative} if the corresponding admissible pair is strongly admissible.
Let $\pi: \r^{n+1} \to \r$ be the projection on the $(n+1)$-th coordinate.
We can associate all the objects defined for an admissible pair to a non-negative semigroup:
\begin{itemize}
\item[-] $\Con(S)$, the cone of the pair $(S,M(S))$.

\item[-] $G(S)$, the group generated by the semigroup $S$.

\item[-] $H_S$, the Hilbert function of the pair $(S,M(S))$.

\item[-] $\Delta (S)$, the Newton-Okounkov convex set of the pair $\Delta (S,M(S))$.

\item[-] $G_0(S) \subset G(S)$, the subgroup $\pi^{-1}(0)\cap G(S)$.

\item[-] $S_k$, the subset $S \cap \pi^{-1}(k)$ of points in $S$ at level $k$.

\item[-] $\ind (S)$, the index of the subgroup $G_0(S)$ in $\z^n \times \{0\}$,
i.e. $\ind(S, \partial M(S))$.

\item[-] $m(S)$, the index $\ind(S, M(S))$.
\end{itemize}

We now give a more refined version of the approximation theorem for the non-negative semigroups.
We will need the following elementary lemma.
\begin{Lem} \label{lem-ball}
Let $B$ be a ball of radius $\sqrt{n}$ centered at a point $a$ in the Euclidean space
$\r^n$ and let $A = B \cap \z^n$. Then: 1) the point $a$ belongs to the convex hull of $A$.
2) The group generated by $x-y$ where $x, y \in A$ is $\z^n$.
\end{Lem}
\begin{proof}
For $n=1$ the statement is obvious. For $n>1$, the lemma
follows from the one-dimensional case and the fact that the ball $B$
contains the product of closed intervals of radius $1$ centered at the projections of
the point $a$ on the coordinate lines.
\end{proof}

\begin{Rem}
K. A. Matveev (an undergraduate student at the University of Toronto) has shown that
the smallest radius for which the above proposition holds is $\sqrt{n+3}/2$.
\end{Rem}

Let us now proceed with the refinement of the approximation theorem for non-negative
semigroups. Let $\dim L(S) = q+1$ and let
$\Con \subset \Con(S)$
be a {closed strongly convex} $(q+1)$-dimensional cone which intersects the boundary
(in the topology of $L(S)$) of $\Con(S)$ only at the origin $0$.

\begin{Th} \label{th-refinement-3.2}
There is a constant $N>0$ (depending on the choice of $\Con$)
such that for any integer $p>N$ which is divisible by $m(S)$ we have: 1) The convex hull of the
set $S_p$ contains the set $\Delta(p)= \Con \cap \pi^{-1}(p)$.
2) The group generated by the differences $x-y$, $x,y \in S_p$ is independent of
$p$ and coincides with the group $G_0(S)$.
\end{Th}
\begin{proof}
By a linear change of variables we can assume that
$L(S)$ is $\r^{q+1}$ (whose coordinates we denote by $x_1, \ldots, x_{q+1}$),
$M(S)$ is the positive half-space
$x_{q+1} \geq 0$, $G(S)$ is $\z^{q+1}$ and the index $m(S)$ is $1$.
To make the notation simpler denote $\Con(S)$ by $\Con_2$.
Take any $(q+1)$-dimensional {closed convex cone} $\Con_1$ such that
1) $\Con \subset \Con_1 \subset \Con_2$ and 2) $\Con_1$ intersects the boundaries of the cones $\Con$ and
$\Con_2$ only at the origin. Consider the sections $\Delta(p)\subset
\Delta_1(p)\subset \Delta_2(p)$ of the cones $\Con \subset \Con_1\subset \Con_2$ (respectively)
by the hyperplane $\pi^{-1}(p)$, for some positive integer $p$. Take $N_1>0$ large enough
so that for any integer $p > N_1$ a ball of radius $\sqrt{q}$ centered at any point of the convex body
$\Delta(p)$ is contained in $\Delta_1(p)$. Then by Lemma \ref{lem-ball} the convex body $\Delta_1(p)$
is contained in the convex hull of the set of integral points in $\Delta_1(p)$.
Also by Theorem \ref{th-semi-gp-approximation} (approximation theorem) there is
$N_2>0$ such that for $p > N_2$
the semigroup $S$ contains all the integral points in $\Delta_1(p)$.
Thus if $p > N = \max\{N_1, N_2\}$, the convex hull of
the set $S_p$ contains the convex body $\Delta(p)$. This proves Part 1).
Moreover, since for $p>N$ $\Delta_1(p)$ contains a ball of radius $\sqrt{q}$ and $S_p$
contains all the integral points in this ball, by
Lemma \ref{lem-ball} the differences of the integral points in
$S_p$ generates the group $\z^q=\z^{q+1}\cap \pi^{-1}(0)$. This proves Part 2).
\end{proof}

\subsection{Hilbert function of a semigroup $S$ and its subsemigroups $\widehat{S}_p$}
Let $S$ be a strongly non-negative semigroup with the Hilbert function $H_S$.
For an integer $p$ in the support of $H_S$ let
$\widehat{S}_p$ denote the subsemigroup generated by $S_p = S \cap \pi^{-1}(p)$.
In this section we compare the asymptotic of $H_S$ with the asymptotic, as $p \to \infty$,
of the Hilbert functions of the semigroups $\widehat{S}_p$.

Later in Sections \ref{subsec-application-valuation} and \ref{subsec-linear-series}
we will apply the
results here to prove a generalization of the Fujita approximation theorem (from the theory of divisors).
Thus we consider the main result
of this section (Theorem \ref{th-1.35}) as an analogue of the Fujita approximation theorem for semigroups.

We will follow the notation introduced in Section \ref{sec-non-negative}. In particular,
$\Delta(S)$ is the Newton-Okounkov body of the semigroup $S$,
$q = \dim \Delta(S)$ its dimension, and $m(S)$ and
$\ind(S)$, the indices associated to $S$.
Also $\Con(\widehat S_{p})$,
$G(\widehat S_{p})$, $H_{\widehat S_{p}}$, $\Delta (\widehat
S_{p})$, $G_0(\widehat S_p)$, $\ind (\widehat S_{p})$, $m(\widehat
S_{p})$, denote the corresponding objects for the semigroup $\widehat{S}_p$.
If $S_p = \emptyset$ put $\widehat S_{p}=\Delta(\widehat
S_{p})=\widehat G_0(S_{p})=\emptyset$ and
$H_{\widehat S_{p}}\equiv 0$.

The next proposition is straightforward to verify:
\begin{Prop}
If the set $S_p$ is not empty then $m(\widehat S_p)=p$,
$\Delta (\widehat S_{p})$ is the convex hull of $S_p$,
the cone $\Con(\widehat S_{p})$ is the cone over $\Delta (\widehat
S_{p})$, $G(\widehat S_{p})$
is the group generated by the set $S_p$ and
$G_0(\widehat S_{p})=G(\widehat S_{p}) \cap \pi^{-1}(0)$ is the group generated by the differences
$a-b$, $a,b\in S_p$. Also $\Con(\widehat S_{p})\subset \Con(S)$.
If $p$ is not divisible by $m(S)$ then $S_p=\emptyset$.
\end{Prop}

Below we deal with functions defined on a non-negative semigroup
$T\subset \z_{\geq 0}$. A semigroup $T \subset \z_{\geq 0}$ contains any large enough integer divisible by $m=m(T)$.
Let $O_m:\z\to \z$ be the scaling map given by $O_m(k)=mk$.
For any function $f: T\to \r$ and for sufficiently large $p$, the pull-back $O_m^*(f)$
is defined by $O_m^*(f)(k) = f(mk)$.

\begin{Def}
Let $\varphi$ be a function defined on a set of sufficiently large natural
numbers. The $q$-th {\it growth coefficient} $a_q(\varphi)$ is the value of the
limit $\lim_{k\to\infty}\varphi(k)/k^q$ (whenever this limit exists).
\end{Def}

The following is a reformulation of Corollary \ref{cor-1.23}.
\begin{Th} \label{th-1.34}
The $q$-th growth coefficient of the function $O_m^*(H_S)$, i.e.
$$a_q(O^*_m(H_S)) = \lim_{k \to \infty} \frac{H_S(mk)}{k^q},$$ exists and is equal to
$\Vol_q(\Delta (S))/\ind(S).$
\end{Th}

For large enough $p$ divisible by $m(S)$,
$S_p \neq \emptyset$ and the subsemigroups $\widehat S_p$ are defined.
The following theorem holds.
\begin{Th} \label{th-1.35}
For $p$ sufficiently large and divisible by $m=m(S)$ we have:
\begin{enumerate}
\item $\dim \Delta(\widehat S_p) =\dim \Delta (S)=q$.
\item $\ind (\widehat S_p)=\ind (S)$.
\item Let the function $\varphi$ be defined by
$$\varphi(p) = \lim_{t \to \infty} \frac{H_{\widehat S_p}(tp)}{t^q}.$$
That is, $\varphi$ is the $q$-th growth coefficient of $O^*_p(H_{\widehat{S}_p})$.
Then the $q$-th growth coefficient of the function $O_m^*(\varphi)$, i.e. $$
a_q(O^*_m(\varphi)) = \lim_{k \to \infty} \frac{\varphi(mk)}{k^q},$$ exists and is equal to
$a_q(O^*_m(H_S)) = \Vol_q(\Delta(S))/\ind(S).$
\end{enumerate}
\end{Th}
\begin{proof}
1) Follows from Theorem \ref{th-refinement-3.2}(1).
2) Follows from Theorem \ref{th-refinement-3.2}(2).
3) By Theorem \ref{th-1.34}, applied to the semigroup
$\widehat S_p$ we have:
$$\varphi(p)= \frac{\Vol_q(\Delta (\widehat S_p))}{\ind(S)}.$$
Now we use Theorem \ref{th-refinement-3.2} to estimate the quantity
$\Vol_q(\Delta(\widehat S_p))$. Let
$\Con_0$ be a $(q+1)$-dimensional {closed cone} contained in $\Con(S)$ which
intersects its boundary (in the topology of the space $L(S)$) only at the origin.
Then, for sufficiently large $p$ and divisible by $m$, the volume $\Vol_q(\Delta (\widehat S_p))$
satisfies the inequalities
$$ \Vol_q(\Con _0\cap \pi ^{-1}(p)) < \Vol_q(\Delta (\widehat S_p))< \Vol_q(\Con(S) \cap \pi ^{-1}(p)).$$
Let $p=km$. Dividing the inequalities above by $k^q \ind(S)$ we obtain
$$\frac { \Vol_q(\Con _0\cap \pi ^{-1}(m))}{\ind(S)} < \frac {\varphi(mk)}{ k^q}<\frac {\Vol_q(\Con(S) \cap \pi ^{-1}(m))}{\ind (S)}=a_q(O^*_m(H_S)).$$
Since we can choose $\Con_0$ as close as we want to $\Con(S)$ this proves Part 3).
\end{proof}


\subsection{Levelwise addition of semigroups}
In this section we define the levelwise addition of non-negative semigroups,
and we consider a subclass of semigroups for which the $n$-th growth
coefficient of the Hilbert function depends on the
semigroup in a polynomial way.

Let $\pi_1:\r^n\times \r \to \r^n$ and $\pi:\r^n\times \r \to \r$ be the projections on
the first and second factors respectively.
Define the operation of levelwise addition $\oplus_t$ on
the pairs of points with the same last coordinate by: $$({\bf x}_1, h)\oplus_t ({\bf x}_2,h) =
({\bf x}_1+{\bf x}_2,h),$$ where $ {\bf x}_1,{\bf x}_2\in \r^n $,
$h\in \r$. In other words, if ${\bf e}$
is the $(n+1)$-th standard basis vector in $\r^n \times \r$ and
${\bf y}_1 = ({\bf x}_1, h)$, ${\bf y}_2 = ({\bf x}_2, h) \in \r^n \times \r$, then
we have ${\bf y}_1\oplus_t {\bf y}_2={\bf y}_1+{\bf y}_2-h{\bf e}$.

Next we define the operation of levelwise addition between any two subsets. Let $X$, $Y \subset \r^n\times \r$.
Then $X\oplus_t Y =Z$ where $Z$
is the set such that for any $h \in \r$
we have: $$\pi_1 (Z \cap \pi^{-1} (h))= \pi_1
(X \cap \pi^{-1} (h))+ \pi_1 (Y \cap \pi^{-1}(h)).$$
(By convention the sum of the empty set with any other set is the empty set.)

The following proposition can be easily verified.
\begin{Prop}
For any two non-negative semigroups
$S_1,S_2$, the set $S=S_1\oplus _t S_2$ is a non-negative semigroup and the following holds:
\begin{enumerate}
\item $L(S)=L(S_1)\oplus_t L(S_2)$.

\item $M(S)=M(S_1)\oplus_t M(S_2)$.

\item $\partial M(S)=\partial M(S_1)\oplus _t \partial M(S_2)$.

\item $G(S)= G(S_1)\oplus _t G(S_2)$.

\item $G_0(S) = G_0(S_1) + G_0(S_2)$.
\end{enumerate}
\end{Prop}

Let us say that a non-negative semigroup has {\it almost all levels} if
$m(S)=1$. Also for a non-negative semigroup $S$, let $\Delta_0(S)$ denote its Newton-Okounkov convex set
shifted to level $0$, i.e. $\Delta_0(S) = \pi_1(\Delta(S))$.
\begin{Prop} \label{prop-levelwise}
For non-negative semigroups
$S_1,S_2$ and $S=S_1\oplus_t S_2$, the following relations hold:
1) The cone $\Con(S)$ is the closure of the levelwise addition
$\Con(S_1)\oplus_t \Con(S_2)$ of the cones $\Con(S_1)$ and $\Con(S_2)$.
2) If the semigroups $S_1, S_2$ have almost all levels then
the Newton-Okounkov set $\Delta(S)$ is the closure of the levelwise addition
$\Delta(S_1) \oplus_t \Delta (S_2)$ of the Newton-Okounkov sets $\Delta(S_1)$ and $\Delta (S_2)$.
(In fact, since in this case the Newton-Okounkov convex sets live in the level $1$, we have that
$\Delta_0(S)$ is the closure of the Minkowski sum $\Delta_0(S_1) + \Delta_0(S_2)$.)
\end{Prop}
\begin{proof}
1) It is easy to see that
$S_1\oplus_t S_2\subset \Con(S_1)\oplus_t \Con(S_2)\subset \Con(S)$
and the set $\Con(S_1)\oplus_t \Con(S_2)$ is dense in $\Con(S)$.
Note that the set $\Con(S_1)\oplus_t \Con(S_2)$ may not be closed (see Example
\ref{ex-2} below). 2) Follows from Part 1). Note that the Minkowski sum of
closed convex subsets may not be closed (see Example \ref{ex-1}).
\end{proof}

\begin{Ex} \label{ex-1}
Let $\Delta_1$, $\Delta_2$
be closed convex sets in $\r^2$ with coordinates $(x,y)$
defined by $\{(x,y) \mid xy\geq 1, y>0\}$ and
$\{(x,y) \mid -xy\geq 1, y>0\}$ respectively.
Then the Minkowski sum $\Delta_1 +\Delta_2$
is the open upper half-plane $\{(x,y) \mid y>0\}$.
\end{Ex}

\begin{Ex} \label{ex-2}
Let $\Delta_1$, $\Delta_2$ be the sets from Example \ref{ex-1}
and let $\Delta_1 \times \{1\}$,
$\Delta_2 \times \{1\}$ in $\r^2 \times \r$ (with coordinates $(x,y, z)$)
be the shifted copies of these sets to the plane $z=1$. Let $\Con_1$ and
$\Con_2$ be the closures of the cones over these sets.
Then $\Con_1\oplus _t \Con_2$ is a non-closed cone which is the union of the set
$\{(x,y,z) \mid 0\leq z, 0< y\}$ and the line $\{(x,y,z) \mid y=z=0\}$.
\end{Ex}

\begin{Prop} \label{prop-strong-levelwise}
Let $S_1$ be a strongly non-negative semigroup and $S_2$ a
non-negative semigroup. Let $S=S_1\oplus S_2$. Then
$\Con(S)= \Con(S_1)\oplus_t \Con(S_2)$ and
$\Reg (S)= \Reg(S_1) \oplus _t \Reg(S_2)$.
If in addition, $S_1, S_2$ have almost all levels, then
$\Delta (S) = \Delta(S_1)\oplus_t \Delta (S_2)$. (In other words,
$\Delta_0(S) = \Delta_0(S_1) + \Delta_0(S_2)$.)
\end{Prop}
\begin{proof}
Let $D$ be the set of pairs $({\bf y}_1,{\bf y}_2) \in \Con(S_1)\times \Con(S_2)$
defined by the condition
$\pi({\bf y}_1)= \pi({\bf y}_2)$. Let us show that the map
$F:D\to \r^n\times \r$ given by
$F({\bf y}_1,{\bf y}_2)= {\bf y}_1 \oplus _t {\bf y}_2$ is proper.
Consider a compact set $K \subset \r^n\times \r$.
The function $x_{n+1}$ is bounded on the compact set $K$, i.e.
there are constants $N_1, N_2$ such that $N_1 \leq x_{n+1}\leq N_2$.
The subset $K_1$ in the cone $\Con(S_1)$ defined by the inequalities
$N_1 \leq x_{n+1}\leq N_2$ is compact.
Consider the set $K_2$ consisting of the points ${\bf y}_2 \in \Con(S_2)$ for which there
is ${\bf y}_1\in K_1$ such that ${\bf y}_1\oplus _t {\bf y}_2 \in K$.
The compactness of $K$ and $K_1$ implies that $K_2$ is also
compact and hence the map $F$ is proper.
The properness of $F$ implies that the sum $\Con(S_1)\oplus_t \Con(S_2)$ is closed which proves
$\Con(S) = \Con(S_1) \oplus_t \Con(S_2)$. The other statements follow from this
and Proposition \ref{prop-levelwise}.
\end{proof}

Finally, let us define $\mathcal{S}(n)$ to be the collection of all
strongly non-negative semigroups $S \subset \z^n\times \z_{\geq 0}$ with almost all levels, i.e. $m(S)=1$, and
$\ind(S)=1$. The set ${\mathcal S}(n)$ is a (commutative) semigroup with respect to the levelwise addition.

%
%

Let $f: \mathcal{S} \to \r$ be a function defined on a (commutative) semigroup $\mathcal{S}$. We say that
$f$ is a {\it homogeneous polynomial of degree $d$} if for any choice of the elements
$a_1, \ldots, a_r \in \mathcal{S}$, the function $F(k_1, \ldots, k_r) = f(k_1a_1+ \cdots + k_ra_r)$,
where $k_1, \ldots, k_r \in \z_{\geq 0}$, is a homogeneous polynomial of degree $d$ in the $k_i$.

\begin{Th}
The function on $\mathcal{S}(n)$ which associates to a semigroup
$S \in \mathcal{S}(n)$, the $n$-th growth coefficient of its Hilbert
function, is a homogeneous polynomial of degree $n$. The value of the polarization of this
polynomial on an $n$-tuple $(S_1, \ldots, S_n)$ is equal to the mixed volume of
the Newton-Okounkov bodies $\Delta(S_1), \ldots, \Delta(S_n)$.
\end{Th}
\begin{proof}
According to Theorem \ref{th-1.34}, the $n$-th growth coefficient of a semigroup
$S \in \mathcal{S}(n)$ exists and is equal to the $n$-dimensional volume of the
convex body $\Delta(S)$. By Proposition \ref{prop-strong-levelwise},
the Newton-Okounkov bodies are added under the levelwise addition of semigroups.
Thus the $n$-th growth coefficient is a homogeneous polynomial of degree $n$
and the value of its polarization is the mixed volume
(see Section \ref{subsec-mixed-vol} for a review of the mixed volume).
\end{proof}

\section{Part II: Valuations and graded algebras}
In this part we consider the graded subalgebras of a polynomial ring in one variable with coefficients
in a field $F$ of transcendence degree $n$ over a ground field $\k$.
For a large class of graded subalgebras (which are not necessarily finitely generated),
we prove the polynomial growth of the Hilbert function,
a Brunn-Minkowski inequality for their growth coefficients and an
{\it abstract} version of the Fujita approximation theorem. We obtain all these from the analogous
results for the semigroups of integral points. The conversion of problems about algebras
into problems about semigroups is made possible via a faithful $\z^n$-valued valuation
on the field $F$. Two sections of this part are devoted to valuations.

\subsection{Prevaluation on a vector space}
In this section we define a prevaluation and discuss its basic properties.
A prevaluation is a weaker version of a valuation which is defined for a vector space
(while a valuation is defined for an algebra).

Let $V$ be a vector space over a field ${\bf k}$
and $I$ a totally ordered set with respect to some ordering $<$.
\begin{Def}
A {\it prevaluation on $V$ with values in $I$} is a function $v: V\setminus \{0\} \to I$ satisfying the following:
\begin{enumerate}
\item For all $f,g \in V$ with $f,g, f+g \neq 0$, we have
$v(f+g) \geq \min(v(f), v(g))$.
\item For all $0 \neq f \in V$ and $0 \neq \lambda \in \k$, $v(\lambda f) = v(f)$.
\end{enumerate}
\end{Def}

\begin{Ex} \label{ex-prevaluation}
Let $V$ be a finite dimensional vector space with a basis $\{e_1, \ldots, e_n\}$ and
$I = \{1, \ldots, n\}$, ordered with
the usual ordering of numbers. For $f = \sum_i \lambda_i e_i$ define
$$v(f) = \min\{ i \mid \lambda_i \neq 0\}.$$ Then $v$ is a prevaluation on $V$ with values in
$I$.
\end{Ex}

Let $v: V \setminus \{0\} \to I$ be a prevaluation.
For $\alpha \in I$, let $V_\alpha = \{ f \in V \mid v(f) \geq \alpha \textup{ or } f = 0\}$.
It follows immediately from the definition of a prevaluation that
$V_\alpha$ is a subspace of $V$.
The {\it leaf} $~\widehat{V}_{\alpha}$ above the point $\alpha \in I$ is the
quotient vector space $V_{\alpha}/\bigcup_{\alpha <\beta}V_{\beta}$.

\begin{Prop} \label{prop-lin-ind}
Let $P \subset V$ be a set of vectors. If
the prevaluation $v$ sends different vectors in $P$ to
different points in $I$ then the vectors in $P$ are linearly independent.
\end{Prop}
\begin{proof}
Let $\sum_{i=1}^s \lambda_i w_i=0$, $\lambda_i\neq 0$,
be a non-trivial linear relation between the vectors in $P$.
Let $\alpha_i = v(w_i)$, $i=1, \ldots, s$, and without loss of generality
assume $\alpha_1<\dots <\alpha _s$.
We can rewrite the linear relation in the form $\lambda_1w_1 =-\sum_{i=2}^s \lambda_i w_i$.
But this cannot hold since $\lambda_1 w_1 \not\in V_{\alpha_{2}}$ while
$\sum_{i>1}\lambda_i w_i\in V_{\alpha_{2}}$.
\end{proof}

\begin{Prop} \label{prop-dim-leaves}
Let $V$ be finite dimensional. Then for all but a finite set of $\alpha\in I$,
the leaf $\widehat V_{\alpha}$ is zero, and we have:
$$\sum_{\alpha \in I} \dim \widehat V_{\alpha} = \dim V.$$
\end{Prop}
\begin{proof}
From Proposition \ref{prop-lin-ind} it follows that $v(V \setminus \{0\})$ contains no more than
$\dim V$ points. Let $v(V\setminus \{ 0\})=\{\alpha_1,\dots,\alpha_{s}\}$ where
$\alpha_1 < \cdots < \alpha_s$.
We have a filtration $V=V_{\alpha_1}
\supset V_{\alpha_2} \supset \cdots \supset V_{\alpha_s}$ and $\dim V$ is
equal to $\sum_{k=1}^{s-1} \dim(V_{\alpha_k}/V_{\alpha_{k+1}}) =\sum_{k=1}^{s-1} \dim \widehat V_{\alpha_k}$.
\end{proof}

Let $W\subset V$ be a non-zero subspace.
Let $J\subset I$ be the image of $W\setminus
\{0\}$ under the prevaluation $v$. The set $J$ inherits a total ordering from $I$.
The following is clear:

\begin{Prop} \label{prop-preval-restriction}
The restriction $v_{|W}: W \setminus \{0\} \to J$ is a prevaluation on $W$.
For each $\alpha \in J$ we have $\dim \widehat{V}_\alpha \geq
\dim \widehat{W}_\alpha$.
\end{Prop}

A prevaluation $v$ is said to have {\it one-dimensional leaves} if
for every $\alpha \in I$ the dimension of the
leaf $\widehat{V}_\alpha$ is at most $1$.

\begin{Prop} \label{prop-dim-Z}
Let $V$ be equipped with an $I$-valued prevaluation $v$ with one-dimensional leaves.
Let $W \subset V$ be a non-zero subspace. Then the number of elements in $v(W \setminus \{0\})$ is equal to $\dim W$.
\end{Prop}
\begin{proof}
Let $J = v(W \setminus \{0\})$. From Proposition \ref{prop-preval-restriction}, $v$ induces a
$J$-valued prevaluation with one-dimensional leaves on the space $W$. The proposition now follows
from Proposition \ref{prop-dim-leaves} applied to $W$.
\end{proof}

\begin{Ex} [Schubert cells in Grassmannian] \label{ex-Grassmannian}
{Let $\k$ be an arbitrary field}.
Let $\textup{Gr}(n, k)$ be the Grassmannian of $k$-dimensional
planes in $\k^n$.  Take the prevaluation $v$ in Example \ref{ex-prevaluation} for $V=\k^n$
and the standard basis. Under this prevaluation each
$k$-dimensional subspace $L \subset \k^n$ goes to a subset $J\subset
I$ with $k$ elements. The set of all $k$-dimensional subspaces
which are mapped onto $J$ forms the {\it Schubert cell $X_J$} in the
Grassmannian $\textup{Gr}(n, k)$.
\end{Ex}

In a similar fashion to Example \ref {ex-Grassmannian}, the Schubert
cells in the variety of complete flags can also be recovered from
the above prevaluation $v$ on $\k^n$.

\subsection{Valuations on algebras} \label{subsec-valuation}
In this section we define a valuation on an algebra and describe its
basic properties. It will allow us to reduce the properties of the Hilbert functions of
graded algebras to the corresponding properties of semigroups. We will present several examples of
valuations.

An {\it ordered abelian group} is an abelian group $\Gamma$ equipped with a
total order $<$ which respects the group operation, i.e. for $a,b,c \in \Gamma$, $a < b$
implies $a+c < b+c$.
\begin{Def}
Let $A$ be an algebra over a field $\k$ and $\Gamma$ an ordered abelian group.
A prevaluation $v: A \setminus \{0\} \to \Gamma$ is a {\it valuation} if, in addition, it satisfies the
following: for any $f, g \in A$ with $f, g \neq 0$, we have $v(fg) = v(f) + v(g).$
The valuation $v$ is called {\it faithful} if its image is the whole $\Gamma$.
\end{Def}

{For the rest of the paper by a valuation we will mean a valuation with one-dimensional leaves.}

\begin{Ex} \label{ex-valuation-curve}
{Let $\k$ be an algebraically closed field, and $X$ an irreducible curve over $\k$.}
As the algebra take the field of rational
functions $\k(X)$ and $\Gamma = \z$ (with the usual ordering of numbers).
Let $a \in X$ be a smooth point.
Then the map
$$v(f) = \ord_a(f)$$ defines a faithful $\z$-valued valuation (with one-dimensional leaves)
on $\k(X)$.
\end{Ex}
The following proposition is straightforward.
\begin{Prop} \label{prop-superadd-subspace}
Let $A$ be an algebra over $\k$  together with a $\Gamma$-valued valuation
$v:A \setminus \{0\} \to \Gamma$. 1) For each subalgebra $B \subset A$,
the set $v(B\setminus \{0\})$ is a subsemigroup of $\Gamma$.
2) For subspaces $L_1, L_2 \subset A$, put $D_1 = v(L_1 \setminus \{0\})$,
$D_2= v(L_2 \setminus \{0\})$ and
$D = v(L_1L_2 \setminus \{0\})$.
We then have $D_1+D_2\subset D$.
\end{Prop}

In general it is not true that $D = D_1+D_2$ as the following example shows.
\begin{Ex} {Let $\k$ be an arbitrary field}, $F=\k(t)$, the field of rational {functions} in one variable, $\Gamma= \z$ (with the usual ordering of numbers) and $v$ the valuation which associates to a polynomial
its order of vanishing at the origin.
Let $L_1 = \span\{ 1, t\}$ and $L_2 = \span\{t, 1+t^2\}$.
Then $D_1=D_2 =\{0, 1\}$. The space $L_1L_2$ is spanned by the polynomials $t$,
$1+t^2$, $t^2$, $t+t^3$, and {hence by $1$, $t$, $t^2$ and $t^3$}.
We have $D = \{0,1,2,3\}$, while $D_1+D_2$ is
$\{0,1,2\}$.
\end{Ex}

We will work with valuations with values in the group $\z^n$ (equipped with some
total ordering). One can define orderings on $\z^n$ as follows. Take $n$ independent linear functions $\ell_1,\dots,\ell_n$
on $\r^n$. For $p, q \in \z^n$ we say that $p > q$ if for some $1 \leq r < n$ we have
$\ell_i(p)=\ell_i(q)$, $i=1 \ldots, r$, and
$\ell_{r+1}(p)> \ell_{r+1}(q)$. This is a total ordering on $\z^n$ which respects the addition.

We are essentially interested in orderings on $\z^n$ whose
restriction to the semigroup $\z^n_{\geq 0}$ is a well-ordering. This holds
for the above ordering if the following properness condition is satisfied:
there is $1 \leq k\leq n$ such that $\ell_1, \ldots, \ell_k$ are non-negative
on $\z^n_{\geq 0}$ and the map $\ell = (\ell_1, \ldots, \ell_k)$ is a proper map from $\z^n_{\geq 0}$
to $\r^k$.

Let us now define the {\it Gr\"{o}bner valuation} on the algebra
$A = \k[[x_1, \ldots, x_n]]$ of formal power series in the variables $x_1, \ldots, x_n$ and with
coefficients in a field $\k$. Fix a total ordering on $\z^n$ (respecting the addition)
which restricts to a well-ordering on $\z^n_{\geq 0}$. For $f \in A$ let
$cx_1^{a_1} \cdots x_n^{a_n}$ be the term in $f$ with the smallest exponent
$(a_1, \ldots, a_n)$ with respect to this ordering. It exists since $\z^n_{\geq 0}$ is
well-ordered. Define $v(f) = (a_1, \ldots, a_n)$. We extend $v$ to the
field of fractions $K$ of $A$ by defining $v(f/g) = v(f) - v(g)$, for any $f, g \in A$, $g \neq 0$.
One verifies that $v$ is a faithful $\z^n$-valued valuation (with one-dimensional leaves) on the
field $K$.

{The following fact is important for us:  {\it any field of transcendence degree $n$ over a ground filed $\k$ 
has faithful $\z^n$-valued valuations} (cf. 
\cite[Chapter 9] {Jacobson}). For all our purposes such valuations
can be realized} as restrictions of the above Gr\"{o}bner valuation to subfields of $K$.


\begin{Ex} \label{ex-valuation-rational-function-alg-var}
Let X be an irreducible $n$-dimensional variety over an
arbitrary field $\k$. First assume that there is a smooth point $p$ in $X$ over $\k$. Let $u_1, \ldots, u_n$  be
regular functions at $p$ which form a system of local coordinates at $p$. Then the field
of rational function on $X$ is naturally embedded in the field of fractions $K$ of the algebra of formal
power series $A = \k[[u_1,\ldots, u_n]]$. The restriction of the above Gr\"{o}bner valuation $v$ to $\k(X)$ gives a
$\z^n$-valued valuation. Since $\k(X)$
contains $u_1,\ldots,u_n$ this valuation is faithful.  In general $X$ may have no
smooth points over $\k$, but almost every point in $X$ over the algebraic
closure $\tilde{\k}$ is a smooth point. Take a smooth point $p$ in $X$ over $\tilde{\k}$. Without loss of
generality we can assume that $X$ is an
affine variety contained in an affine space $\mathbb{A}^N$, moreover we can assume
that the projection of $X$ to the coordinate plane with coordinates $x_1,\ldots, x_n$ is non-degenerate
at the point $p$. Then the functions $u_i = x_i - \alpha_i$ where $\alpha_i = x_i(p)$, $i = 1,\ldots, n$, form
a local system of coordinates at $p$. Consider a $\z^n$-valued faithful valuation $v$ on $\tilde{\k}(X)$ constructed as
above from a Gr\"{o}bner valuation on the algebra $\tilde{\k}[[u_1,\ldots, u_n]]$. For each $\alpha_i = x_i(p)$
choose a polynomial $g_i$ in $x_i$ with coefficients in $\k$ which has $\alpha_i$ as a root. The vectors $v(u_i)$ 
generate the lattice $\z^n$ and we have $v(g_i)= \ord_{\alpha_i}(g_i) v(u_i)$. The image of $\k(X) \setminus \{0\}$
under the valuation $v$ contains the vectors $v(g_i)$ and hence is a
sublattice $\Lambda$ of rank $n$ in $\z^n$. The restriction of the valuation $v$ to $\k(X)$ is a faithful 
$\Lambda$-valued valuation, and $\Lambda$ is isomorphic to $\z^n$ as a group.
\end{Ex}

Let $Y$ be an irreducible variety
birationally isomorphic to $X$.
Then a valuation $v$ on the field of rational functions on
$Y$ (e.g. the faithful $\z^n$-valued valuation in Example \ref{ex-valuation-rational-function-alg-var})
automatically gives a valuation on the field of rational functions on $X$.
The following is an example of this kind of valuation defined in terms of
the variety $X$, although it indeed corresponds to a system of parameters at a smooth point
on some birational model $Y$ of $X$ (at least when the ground field $\k$ {is algebraically closed and} has characteristic $0$).

\begin{Ex}[Valuation constructed from a Parshin point on $X$] \label{ex-valuation-Parshin}
Let $X$ be an irreducible  $n$-dimensional variety {over an algebraically closed field $\k$}. Consider a sequence of
maps $$X_\bullet: \quad \{a\}=X_0 \stackrel{\pi_0}{\to}  X_1 \stackrel{\pi_1}{\to} \cdots
\stackrel{\pi_{n-1}}{\to} X_n \stackrel{\pi_n}{\to} X,$$ where
each $X_i$ is a normal irreducible variety of dimension $i$,
the map $X_i \stackrel{\pi_i}{\to} X_{i+1}$ for $i=0,\ldots, n-1$, is the normalization map for the image
$\pi_i(X_i)\subset X_{i+1}$ { and $\pi_n$ is the normalization map for $X$}. We call such a sequence $X_\bullet$ a {\it
Parshin point} on the variety $X$. We say that a collection of rational functions
$f_1,\dots,f_n$ is {\it a system of parameters} about
$X_\bullet$, if for each $i$, the function $\pi^*_i\circ \dots\circ \pi^*_{n}(f_i)$
vanishes at order $1$ along the
hypersurface  $\pi_{i-1}(X_{i-1})$ in the normal variety $X_i$.
Given a Parshin point $X_\bullet$ together with a system of parameters,
one can associate an {\it iterated Laurent series} to each rational function
$g$ on $X$ (see \cite{Parshin, Okounkov-log-concave}).
An iterated Laurent series is defined inductively (on the number of parameters). It is a usual
Laurent series  $\sum_{k} c_kf_n^k$ with a finite number of terms with negative degrees in the variable
$f_n$ and every coefficient $c_k$
is an iterated Laurent series in the variables
$f_1,\dots,f_{n-1}$. An iterated Laurent series has a monomial
$cf_1^{k_1}\dots f_n^{k_n}$ of smallest degree with respect to
the lexicographic order in the degrees $(k_1,\dots,k_n)$ (where we order the parameters
by $f_n > f_{n-1} > \cdots > f_1$).
The map which assigns to a Laurent series its smallest monomial
defines a faithful valuation (with one-dimensional leaves) on the field of rational functions on
$X$.
\end{Ex}

{Finally let us give an example of a faithful $\z^n$-valued valuation on a field of transcendence degree $n$
over a ground field $\k$ which is not finitely generated over $\k$.}

\begin{Ex} \label{ex-2.26}
{Let $\k$ be an arbitrary field. As above let $K$ denote the field of fractions of the
algebra of formal power series $\k[[x_1, \ldots, x_n]]$} and let $K'$ be the subfield
consisting of the elements which are algebraic over the field of rational functions
$\k(x_1, \ldots, x_n)$. This subfield has transcendence degree $n$ over $\k$ but is not finitely generated
over $\k$. The restriction of the above Gr\"{o}bner valuation on $K$ to $K'$ gives a faithful valuation
on $K'$
\end{Ex}

\subsection{Graded subalgebras of the polynomial ring $F[t]$} \label{subsec-graded-alg}
In this section we introduce certain large classes of graded algebras and
discuss their basic properties.

{Let $F$ be a field containing a field $\k$ which we take as the ground field.}
A homogenous element of degree $m\geq 0$ in $F[t]$ is an element
$a_mt^m$ where $a_m\in F$. (For any $m$ the element $0\in F$
is a homogeneous element of degree $m$.)
Let $M$ be a linear subspace of $F[t]$.
For any $k\geq 0$  the collection $M_k$
of homogeneous elements of degree $k$ in $M$ is a linear subspace over $\k$
called the {\it $k$-th homogenous component of $M$}. Similarly
the linear subspace $L_k \subset F$ consisting of those $a$ such that $at^k\in M_k$
is called the {\it $k$-th subspace of $M$}. A linear
subspace $M \subset F[t]$
over $\k$ is called a {\it graded space} if it is the direct sum of its homogeneous
components. A subalgebra  $A\subset F[t]$
is called {\it graded} if it is graded as a linear subspace of $F[t]$.

We now define three classes of graded subalgebras which will play
main roles later:
\begin{enumerate}
\item To each non-zero finite dimensional linear subspace
$L \subset F$ over $\k$ we associate the graded algebra $A_L$
defined as follows: its zero-th homogeneous component is $\k$ and for each $k > 0$ its
$k$-th subspace is $L^k$, the subspace
spanned by all the products $f_1\cdots f_k$ with $f_i \in L$.
That is, $$A_L = \bigoplus_{k \geq 0} L^kt^k.$$
The algebra $A_L$ is a graded algebra generated by $\k$ and finitely many
elements of degree $1$.

\item We call a graded subalgebra $A \subset F[t]$,
an {\it algebra of integral type}, if there is an algebra $A_L$,
for some non-zero finite dimensional subspace $L$ over $\k$,
such that $A$ is a finitely generated $A_L$-module. ({Equivalently, if
$A$ is finitely generated over $\k$ and is a finite module over the subalgebra generated by $A_1$}.)

\item  We call a graded subalgebra $A \subset F[t]$,
an {\it algebra of almost integral type}, if there is an algebra $A' \subset F[t]$ of integral type
such that $A\subset A'$. (Equivalently if $A \subset A_L$
for some finite dimensional subspace $L \subset F$.)
\end{enumerate}

{As mentioned in the introduction, algebras $A_L$ {are related to the} homogeneous coordinate rings
of projective varieties, algebras of integral type are related to the rings of sections of ample line bundles and
algebras of almost integral type to the rings of sections of arbitrary line bundles
(see Theorems \ref{th-ring-sec-almost-finite-type} and \ref{th-very-ample-int-closure}).}

As the following shows, the class of algebras of almost integral type already
contains the class of finitely generated graded subalgebras. Although, in general,
an algebra of almost integral type may not be finitely generated.
\begin{Prop}
Let $A$ be a finitely generated graded subalgebra of $F[t]$ (over $\k$).
Then $A$ is an algebra of almost integral type.
\end{Prop}
\begin{proof}
Let $f_1t^{d_1}, \ldots, f_rt^{d_r}$ be a set of homogeneous generators for $A$.
Let $L$ be the subspace spanned by $1$ and all the $f_i$. Then
$A$ is contained in the algebra $A_L$ and hence is of almost integral type.
\end{proof}

The following proposition is easy to show.
\begin{Prop}\label{Prop-finite-module}
Let $M \subset F[t]$ be a graded subspace and write $M = \bigoplus_{k\geq 0} L_kt^k$, where $L_k$ is the
$k$-th subspace of $M$.
Then $M$ is a finitely generated module over an algebra $A_L$ if
and only if there exists $N>0$ such that for any $m \geq N$ and $\ell > 0$ we have $L_{m+\ell} = L_m L^{\ell}$.
\end{Prop}

Let $A$ be a graded subalgebra of $F[t]$. Let us denote the
integral closure of $A$ in the field of fractions $F(t)$ by $\ol{A}$.
\footnote{Let $A \subset B$ be commutative rings. An element $f \in B$ is called {\it integral} over $A$ if
$f$ satisfies an equation $f^m + a_1f^{m-1} + \cdots + a_m = 0,$ for $m>0$ and $a_i \in A$, $i=1, \ldots, m$.
The integral closure $\ol{A}$ of $A$ in $B$ is the collection of all the elements of $B$ which are
integral over $A$. It is a ring containing $A$.}
It is a standard result that $\ol{A}$ is contained in $F[t]$ and is
graded (see \cite[Ex. 4.21]{Eisenbud}).

The following is a corollary of the classical theorem of Noether on finiteness of integral closure.
\begin{Th} \label{th-A-bar-finite-type}
Let $F$ be a finitely generated field over a field ${\bf k}$ and $A$ a graded subalgebra of $F[t]$.
1) If $A$ is of integral type then $\ol{A}$ is also of integral type.
2) If $A$ is of almost integral type then $\ol{A}$ is also of almost integral type.
\end{Th}

Let $L \subset F$ be a linear subspace over $\k$. Let $P(L) \subset F$ denote the field
consisting of all the elements $f/g$ where $f, g \in L^k$ for some $k>0$ and $g \neq 0$.
We call $P(L)$, the {\it subfield associated to $L$}, and its transcendence degree over $\k$,
the {\it projective transcendence degree} of the subspace $L$.

\begin{Def}
The {\it Hilbert function} of a graded subspace $M \subset F[t]$ is the function $H_M$
defined by $H_M(k) = \dim M_k$ (over $\k$), where $M_k$ is the $k$-th homogeneous component of
$M$. We put $H_M(k)=\infty$ if $M_k$ is infinite dimensional.
\end{Def}

The theorem below is a corollary of
the so-called Hilbert-Serre theorem on Hilbert function of a
finitely generated module over a polynomial ring.
Algebraic and combinatorial proofs of this theorem can be found in
\cite[Chap. VII, \S 12]{Zariski}, \cite{Askold-finite-sets} and \cite{Askold-Chulkov}.

\begin{Th} \label{th-Hilbert-th-module}
Let $L \subset F$ be a finite dimensional subspace over $\k$ and let $q$ be its
projective transcendence degree. Let $M \subset F[t]$ be a finitely generated graded module over $A_L$.
Then for sufficiently large values of $k$, the Hilbert function $H_M(k)$ of $M$ coincides with
a polynomial $\tilde{H}_M(k)$ of degree $q$. The leading coefficient of this polynomial
multiplied by $q!$ is a positive integer.
\end{Th}

\begin{Def}
The polynomial $\tilde{H}_M$ in Theorem \ref{th-Hilbert-th-module} is called the {\it Hilbert polynomial} of
the graded module $M$.
\end{Def}

Two numbers appear in Theorem \ref{th-Hilbert-th-module}: the degree $q$ of the Hilbert polynomial and its
leading coefficient multiplied by $q!$. When $M = A_L$ both of these numbers have geometric meanings
(see Section \ref{sec-dim-deg-proj-var}).

Assume that a graded algebra $A \subset F[t]$ has at least one non-zero homogeneous component of
positive degree. Then the set of $k$ for which the homogeneous component $A_k$ is not
$0$ forms a non-trivial semigroup $T \subset \z_{\geq 0}$. Let $m(A)$ be the index of the group
$G(T)$ in $\z$. When $k$ is sufficiently large, the homogeneous component $A_k$ is non-zero
(and hence $H_A(k)$ is non-zero) if and only if $k$ is divisible by $m(A)$.
{It follows from definition that when $A \subset F[t]$ is of integral type we have $m(A) = 1$.}


Next we define the componentwise product of graded spaces. Recall that
for two subspaces $L_1$, $L_2 \subset F$, the product $L_1L_2$ denotes
the $\k$-linear subspace spanned by all the products $fg$, where $f \in L_1$, $g \in L_2$.

\begin{Def} \label{def-prod-subspace}
The collection of all the non-zero finite dimensional subspaces of $F$ is a (commutative) semigroup with respect
to this product. We will denote it by ${\bf K}(F)$.
\end{Def}

\begin{Def} \label{def-componentwise}
Let $ M',M''$ be graded spaces with $k$-th subspaces $L'_k$, $L''_k$ respectively.
The {\it componentwise product} of spaces $M'$ and $M''$ is the graded space $M = M'M''$
whose $k$-th subspace $L_k$ is $L'_k L''_k$.
\end{Def}

In particular, the componentwise product can be applied to graded subalgebras of $F[t]$.
The following can be easily verified:
\begin{Prop} \label{prop-5.2}
1) The componentwise product of graded algebras is a graded algebra.
2) Let $L', L'' \subset F$ be two non-zero finite dimensional subspaces over $\k$
and let $L = L'L''$. Then $A_L = A_{L'}A_{L''}$. 3) Let $M', M''$ be two
finitely generated modules over $A_{L'}$ and $A_{L''}$ respectively. Then $M=M'M''$ is
a finitely generated module over $A_L$ where $L = L'L''$.
4) If $A',A''$ are algebras of integral type (respectively of almost integral type) then
$A=A'A''$ is also of integral type (respectively of almost integral type).
\end{Prop}


\begin{Cor}
1) The map $L \mapsto A_L$ is an isomorphism between the semigroup ${\bf K}(F)$ of non-zero
finite dimensional subspaces in $F$, and the semigroup of subalgebras $A_L$ with respect to
the componentwise product. 2) The collection of algebras of almost
integral type in $F[t]$ is a semigroup with respect to the componentwise product of subalgebras.
\end{Cor}

\subsection{Valuations on graded algebras and semigroups} \label{subsec-application-valuation}
In this section given a valuation on the field $F$ we construct a valuation on
the ring $F[t]$. Using this valuation we will deduce results about the graded
algebras of almost integral type from the analogous results for
the strongly non-negative semigroups.

It will be easier to prove the statements in this section
if in addition $F$ is assumed to be finitely generated over ${\bf k}$.
One knows that {\it a field extension $F/{\bf k}$ is finitely generated if and only if it is the field of rational functions of an irreducible algebraic variety over ${\bf k}$}. Moreover, {\it the transcendence degree of
$F/{\bf k}$ is the dimension of the variety $X$}. The following simple proposition justifies that the
general case can be reduced to the case where $F$ is finitely generated over ${\bf k}$.

\begin{Prop}
Let $A_1, \ldots, A_k \subset F[t]$ be algebras of almost integral type over ${\bf k}$. Then there exists
a field $F_0 \subset F$ which is finitely generated over ${\bf k}$ such that $A_1, \ldots, A_k
\subset F_0[t]$. If $F$ has finite transcendence degree over ${\bf k}$ then the field $F_0$ can be chosen to have the same transcendence degree.
\end{Prop}

Thus to prove a statement about a finite collection of subalgebras $A_1, \ldots, A_k \subset F[t]$ of almost integral type
over $\k$, it is enough to prove it for the case where $F$ is finitely generated over $\k$.

{To carry out our constructions we need a faithful valuation $v: F \setminus \{0\} \to \z^n$ with one-dimensional leaves (where it is understood that $\z^n$ is equipped with a total order $<$ respecting addition).
By the above proposition, we can assume that $F$ is finitely generated over $\k$, i.e. is the field of rational
functions on a variety. In this case, one can construct many such valuations
(see Examples \ref{ex-valuation-rational-function-alg-var} and \ref{ex-valuation-Parshin}).}

{Using $v$ on $F$ we define a $\z^n \times \z$-valued valuation $v_t$ on the algebra $F[t]$.}
Consider the total ordering $\prec$ on the group $\z^n \times \z$ given by the following:
let $(\alpha,n), (\beta,m)\in \z^n \times \z$.
\begin{enumerate}
\item If $n>m$ then $(\alpha,n) \prec (\beta,m)$.
\item If $n=m$ and $\alpha <\beta$ then $(\alpha,n)\prec (\beta,m)$.
\end{enumerate}

\begin{Def}
Define $v_t: F[t] \setminus \{0\} \to \z^n \times \z$ as follows:
Let $P(t) = a_nt^n+\cdots+a_0$, $a_n\neq 0$, be a polynomial in $F[t]$.
Then $$v_t(P)=(v(a_n), n).$$
It is easy to verify that $v_t$ is a valuation (extending $v$ on
$F$) where $\z^n \times \z$ is equipped with the total ordering $\prec$. The extension of
$v_t$ to the field of fractions $F(t)$ is faithful and has one-dimensional leaves.
\end{Def}

Let $A \subset F[t]$ be a graded subalgebra.
Then $$S(A) = v_t(A \setminus \{0\}),$$
is a non-negative semigroup (see Proposition \ref{prop-superadd-subspace}).
We will use the following notations:
\begin{itemize}
\item[-] $\Con(A)$, the cone of the semigroup $S(A)$,

\item[-] $G(A)$, the group generated by the semigroup $S(A)$,

\item[-] $G_0(A)$, the subgroup $G_0(S(A))$.

\item[-] $H_A$, the Hilbert function of the graded algebra $A$,

\item[-] $\Delta(A)$, the Newton-Okounkov convex set of the semigroup $S(A)$,


\item[-] $m(A)$, $\ind(A)$, the indices $m(S(A))$, $\ind(S(A))$ for the semigroup $S(A)$ respectively.

\end{itemize}

\begin{Prop}
The Hilbert function $H_{S(A)}$ of the non-negative semigroup $S(A)$
coincides with the Hilbert function $H_A$ of the algebra $A$.
\end{Prop}
\begin{proof}
Follows from Proposition \ref{prop-dim-Z}.
\end{proof}

Now we show that when $A$ is an algebra of almost integral type then the semigroup $S(A)$ is
strongly non-negative.
\begin{Lem} \label{lem-S(A)-strongly-non-negative}
Let $A$ be an algebra of integral type. Assume that the rank of
$G(A)\subset \z^n\times \z$ is equal to $n+1$.
Then the semigroup $S(A)$ is strongly non-negative.
\end{Lem}
\begin{proof}
It is obvious that the semigroup $S(A)$ is non-negative.
Let $A$ be a finitely generated module over some algebra $A_L$.
Since $P(L) \subset F$, the projective transcendence degree of $L$ cannot be bigger than
$n$. By Theorem \ref{th-Hilbert-th-module} (Hilbert-Serre theorem), for large values of $k$,
the Hilbert function of the algebra $A$ is a polynomial in $k$ of degree $\leq n$.
Thus by Theorem \ref{th-5.4} the semigroup $S(A)$ is strongly non-negative.
\end{proof}

\begin{Lem} \label{lem-G(A)-full}
Let $A$ be an algebra of integral type.
Then there exists an algebra of integral type $B$ containing $A$ such that
the group $G(B)$ is the whole $\z^n\times \z$.
\end{Lem}
\begin{proof}
By assumption $v$ is faithful and hence we can find elements $f_1, \ldots, f_n \in F$
such that $v(f_1), \ldots, v(f_n)$ is the standard basis for $\z^n$. Consider the space
$L$ spanned by $1$ and $f_1, \ldots, f_n$ and take its associated graded algebra $A_L$.
The semigroup $S(A_L)$ contains the basis
$\{{\bf e}_{n+1}, {\bf e}_1 + {\bf e}_{n+1}, \ldots, {\bf e}_n + {\bf e}_{n+1}\}$,
where $\{{\bf e}_1, \ldots, {\bf e}_{n+1}\}$ is the standard basis in $\r^n \times \r = \r^{n+1}$.
Hence $G(A_L) = \z^n \times \z$. Let $B = A_LA$ be the componentwise product of $A$ and $A_L$.
One sees that $G(B) = \z^n \times \z$. Since $1 \in L$ we have $A \subset B$.
\end{proof}

\begin{Th} \label{th-S(A)-strongly-non-neg}
Let  $A\subset F[t]$ be an algebra of almost integral type.
Then $S(A)$ is a strongly non-negative semigroup, and hence its Newton-Okounkov
convex set $\Delta(A)$ is a convex body.
\end{Th}
\begin{proof}
By definition the algebra $A$
is contained in some algebra of integral type, and moreover
by Lemma \ref{lem-G(A)-full}, it is contained in an algebra $B$ of integral type such that
$G(B)=\z^n\times \z$. By Lemma \ref{lem-S(A)-strongly-non-negative},
$S(B)$ is strongly non-negative. Since $A\subset B$ we have
$S(A) \subset S(B)$ which shows that $S(A)$ is also strongly non-negative.
\end{proof}

Using Theorem \ref{th-S(A)-strongly-non-neg} we can translate the results in Part I
about the Hilbert function of strongly non-negative semigroups to results about the Hilbert function of
algebras of almost integral type.

Let $A$ be an algebra of almost integral type with the Newton-Okounkov body $\Delta(A)$.
Put $m=m(A)$ and $q=\dim \Delta(A)$. The Hilbert function $H_A$ vanishes at those $p$ which are not divisible by $m$.
Recall that $O_m$ denotes the scaling map $O_m(k)=mk$. For a function $f$, $O^*_m(f)$ is the
pull-back of $f$ defined by $O^*_m(f)(k) = f(mk)$, for all $k$. Also $\Vol_q$ denotes the integral volume
(Definition \ref{def-int-volume}).

\begin{Th} \label{th-growth-alg-Newton-Okounkov-body}
The $q$-th growth coefficient of the function $O_m^*(H_A)$, i.e.
$$a_q(O^*_m(H_A)) = \lim_{k \to \infty} \frac{H_A(mk)}{k^q},$$ exists and is equal to
$\Vol_q(\Delta(A))/\ind(A)$.
\end{Th}
\begin{proof}
This follows from Theorem \ref{th-S(A)-strongly-non-neg} and Theorem \ref{th-1.34}.
\end{proof}

The semigroup associated to an algebra of almost integral type has
the following superadditivity property with respect to the componentwise product:

\begin{Prop} \label{prop-superadd-alg}
Let $A', A''$ be algebras of almost integral type and $A=A' A''$. Put
$S=v_t(A\setminus \{0\})$, $S'=v_t(A'\setminus \{0\}))$ and $S''=v_t(A''\setminus \{0\})$.
Then  $S' \oplus_t S'' \subset S$.
Moreover, if $m(A')=m(A'')=1$ then $$\Delta(A') \oplus_t \Delta(A'') \subset \Delta(A).$$
(In other words, $\Delta_0(A') + \Delta_0(A'') \subset \Delta_0(A)$, where
$\Delta_0$ is the Newton-Okounkov body projected to the level $0$ and $+$ is the Minkowski sum.)
\end{Prop}
\begin{proof}
If $L_k$, $L'_k$ and $L''_k$ are the $k$-th subspaces corresponding to
$A, A'$ and $A''$ respectively, then by definition $L_k= L'_{k}L''_{k}$.
According to Proposition \ref{prop-superadd-subspace} we have
$v( L'_{k}\setminus \{0\}) + v( L''_{k}\setminus \{0\}) \subset
v( L_{k}\setminus \{0\})$. The proposition follows from this inclusion.
\end{proof}

Next we prove a Brunn-Minkowski type inequality for the $n$-th growth coefficients
of Hilbert functions of algebras of almost integral type, where as usual $n$ is the
transcendence degree of $F$ over $\k$. This is a generalization of
the corresponding inequality for the volume of big divisors (see Corollary \ref{cor-linear-series}(3)
and Remark \ref{rem-linear-series}).

\begin{Th} \label{th-Brunn-Mink-algebra}
Let $A_1$, $A_2$ be algebras of almost integral type and let $A_3 = A_1A_2$ be their
componentwise product. Moreover assume $m(A_1) = m(A_2) = 1$. Then the $n$-growth
coefficients $\rho_1, \rho_2$ and $\rho_3$ of the Hilbert functions of the algebras
$A_1, A_2, A_3$ respectively, satisfy the following Brunn-Minkowski type inequality:
\begin{equation} \label{eqn-Brunn-Mink-algebra}
\rho_1^{1/n}+\rho_2^{1/n} \leq  \rho_3^{1/n}.
\end{equation}
\end{Th}
\begin{proof}
By Proposition \ref{prop-superadd-alg} applied to the valuation $v_t$,
we have $S(A_1) \oplus_t S(A_2) \subset S(A_3)$, and
$\Delta(A_1) \oplus_t \Delta(A_2) \subset \Delta(A_3)$.
From the classical Brunn-Minkowski inequality (Theorem \ref{th-Brunn-Mink})
we then get
\begin{equation} \label{equ-Brunn-Mink-proof-alg}
\Vol_n^{1/n}(\Delta(A_1)) + \Vol_n^{1/n}(\Delta(A_2)) \leq \Vol_n^{1/n}(\Delta(A_3)).
\end{equation}
For $i=1,2,3$, we have $\rho_i = \Vol_n(\Delta(A_i)) / \ind(A_i)$ (Theorem \ref{th-growth-alg-Newton-Okounkov-body}).
Since $S(A_1) \oplus_t S(A_2) \subset S(A_3)$, the index $\ind(A_3)$ is less than or equal to both of the
indices $\ind(A_1)$ and $\ind(A_2)$. From this and (\ref{equ-Brunn-Mink-proof-alg}) the required inequality (\ref{eqn-Brunn-Mink-algebra})
follows.
\end{proof}

Let $A\subset F[t]$ be an algebra of almost integral type.
For an integer $p$ in the support of the Hilbert function $H_A$
let $\widehat{A}_p$ be the graded subalgebra generated by the $p$-th homogeneous
component $A_p$ of $A$. We wish to compare the asymptotic of $H_A$ with the
asymptotic, as $p$ tends to infinity, of the growth coefficients of the Hilbert
functions of the algebras $\widehat{A}_p$.

To every $p$ in the support of $H_A$ we
associate two semigroups: 1) the semigroup $\widehat{S}_p(A)$
generated by the set $S_p(A)$ of points at level $p$ in $S(A)$, and
2) The semigroup $S(\widehat{A}_p)$ associated to the algebra $\widehat{A}_p$.

\begin{Th} \label{th-5.14}
Let $A$ be an algebra of almost integral type and $p$ any integer in the support of
$H_A$. Then the semigroup
$S(\widehat{A}_p)$ satisfies the inclusions:
$$ \widehat{S}_p (A)\subset S( \widehat{A}_p)\subset S(A).$$
\end{Th}
\begin{proof}
The inclusion $S(\widehat{A}_p)\subset S(A)$ follows from
$\widehat A_p\subset A$. By definition, the set of points at level $p$ in the semigroups
$\widehat S_p(A)$ and $S(\widehat{A}_p)$ coincide. Denote this set by $S_p$.
For any $k > 0$, the set of points in $ \widehat{S}_p(A)$ at the level $kp$ is equal to
$k*S_p= S_p + \cdots + S_p$ ($k$-times), and the set $S_{kp}(\widehat{A}_p)$ is equal to $v_t (A_p^k
\setminus \{0\})$. By Proposition \ref{prop-superadd-alg}
we get $k*S_p \subset v_t(A_p^k\setminus \{0\})$, i.e.
$k*S_p \subset S_{kp}(\widehat{A}_p)$, which implies the required inclusion.
\end{proof}

Let $A$ be an algebra of almost integral type with index $m = m(A)$. Any positive integer
$p$ which is sufficiently large and is divisible by $m$ lies in the support of
the Hilbert function $H_A$ and hence the subalgebra
$\widehat{A}_p$ is defined. To this subalgebra there corresponds its Hilbert function
$H_{\widehat{A}_p}$, the semigroup $S(\widehat{A}_p)$, the Newton-Okounkov body $\Delta(\widehat{A}_p)$,
and the indices $m(\widehat{A}_p)$, $\ind(\widehat{A}_p)$.

The following can be considered as a
generalization of the Fujita approximation theorem (regarding the volume of big divisors)
to algebras of almost integral type.

\begin{Th} \label{th-Fujita-algebra}
For $p$ sufficiently large and divisible by $m=m(A)$ we have:
\begin{enumerate}
\item $\dim \Delta (\widehat A_p) =\dim \Delta(A)=q$.
\item $\ind (\widehat A_p)=\ind (A)$.
\item Let the function $\varphi$ be defined by
$$\varphi(p) = \lim_{t \to \infty} \frac{H_{\widehat A_p}(tp)}{t^q}.$$
That is, $\varphi$ is the $q$-th growth coefficient of $O^*_p(H_{\widehat{A}_p})$.
Then the $q$-th growth coefficient of the function $O_m^*(\varphi)$, i.e. $$
a_q(O^*_m(\varphi)) = \lim_{k \to \infty} \frac{\varphi(mk)}{k^q},$$ exists and is equal to
$a_q(O^*_m(H_A)) = \Vol_q(\Delta(A))/\ind(A).$
\end{enumerate}
\end{Th}
\begin{proof}
This follows from Theorem \ref{th-5.14} and Theorem \ref{th-growth-alg-Newton-Okounkov-body}.
\end{proof}

When $A$ is an algebra of integral type, Theorem \ref{th-Fujita-algebra} can be refined
using the Hilbert-Serre theorem (Theorem \ref{th-Hilbert-th-module}).
Note that when $A$ is of integral type have $m(A)=1$.
\begin{Th} \label{th-5.16}
Let $A$ be an algebra of integral type and, as in Theorem
\ref{th-Fujita-algebra}, let $\varphi(p)$ be the $q$-th growth coefficient of
$O^*_p(H_{\widehat{A}_p})$. Then for sufficiently large $p$, the number $\varphi(p)/p^q$
is independent of $p$ and we have:
$$\frac{\varphi(p)}{p^q}= \frac{\Vol_q(\Delta (A))}{\ind(A)}=a_q(H_A).$$
\end{Th}
\begin{proof}
This follows from Theorem \ref{th-Hilbert-th-module}.
Let $\tilde H_A(k)=a_qk^q +\dots+a_0 $ be the Hilbert polynomial of the algebra $A$.
From Proposition \ref{Prop-finite-module}
it follows that if $p$ is sufficiently large then,
for any $k>0$, the $(kp)$-th homogeneous component of the algebra $\widehat A_p$
coincides with the $(kp)$-th homogeneous component of the algebra $A$,
and hence the dimension of the $(kp)$-th homogeneous component of
$\widehat A_p$ is equal to $\tilde{H}_A(kp)$. Thus the $q$-th growth coefficient of the function
$O^*_p (H_{\widehat A_p})$ equals $p^q a_q$ which proves the theorem.
\end{proof}

\section{Part III: Projective varieties and algebras of almost integral type}
The famous Hilbert theorem computes the dimension and degree of
a projective subvariety of projective space by means of the asymptotic growth of
its Hilbert function. The constructions and results in the previous parts relate
the asymptotic of Hilbert function with
the Newton-Okounkov body. In this part we use Hilbert's theorem to
give geometric interpretations of these results.
{We will take the ground field $\k$ to be algebraically closed.}

\subsection{Dimension and degree of projective varieties} \label{sec-dim-deg-proj-var}
In this section we give a geometric interpretation of the dimension and degree
of (the closure of) the image of an irreducible variety under a rational
map to projective space.


{Let $X$ be an irreducible algebraic variety over $\k$ of dimension $n$, and let $F = \k(X)$ denote the
field of rational functions on $X$.}
{Recall from the introduction that to each non-zero finite dimensional subspace $L \subset F$ we associate
the Kodaira rational map $\Phi_L: X \ratmap \p(L^*)$.}
Let $Y_L$ denote the closure of the image of $X$ under the Kodaira map in $\p(L^*)$ (more precisely,
the image of a Zariski open subset of $X$ where $\Phi_L$ is defined).

Consider the algebra $A_L$ associated to $L$.
For large values of $k$,
the Hilbert function $H_{A_L}(k)$ coincides with the Hilbert polynomial
$\tilde{H}_{A_L}(k) = a_qk^q + \cdots + a_0$.
The following is the celebrated Hilbert theorem on the dimension and
degree of a projective subvariety, customized for the purposes of this paper
{(see \cite[Section I.7]{Hartshorne}).}
\begin{Th}[Hilbert] \label{th-Hilbert}
The degree $q$ of the Hilbert polynomial $\tilde{H}_{A_L}$ is equal to the dimension of the
variety $Y_L$, and its leading coefficient $a_q$ multiplied by $q!$ is equal to the degree of the subvariety
$Y_L$ in the projective space $\p(L^*)$.
\end{Th}

Fix a faithful $\z^n$-valued valuation $v$ on
the field of rational functions $F = \k(X)$.
The extension $v_t$ of $v$ to $F[t]$, associates to any algebra $A$ of
almost integral type the strongly non-negative semigroup
$S(A) \subset \z^n\times \z_{\geq 0}$.
Comparing Theorem \ref{th-5.1} and Theorem \ref{th-Hilbert} (Hilbert's theorem)
we obtain that for $A=A_L$, the Newton-Okounkov body $\Delta(A_L)$ is
responsible for the dimension and degree of
the variety $Y_L$.

\begin{Cor} \label{cor-deg-Y_L-vol-Delta}
The dimension $q$ of the Newton-Okounkov body $\Delta(A_L)$ is equal to
the dimension of the variety $Y_L$, and its $q$-dimensional
integral volume $\Vol_q(\Delta(A_L))$ multiplied by $q!/\ind(A_L)$ is equal to
the degree of $Y_L$.
\end{Cor}

Let $A$ be an algebra of almost integral type in $F[t]$.
Let $L_k$ be the $k$-th subspace of the algebra $A$. To
each non-zero subspace $L_k$ we can associate the following objects:
the Kodaira map $\Phi_{L_k}:X\ratmap \p(L_k^*)$, the variety $Y_{L_k}\subset \p(L_k^*)$
(i.e. the closure of the image of $\Phi_{L_k}$) and its dimension and degree.
Recall that for a sufficiently large integer $p$ divisible by
$m=m(A)$, the space $L_p$ is non-zero. As before let $O_m$ be the scaling map
$O_m(k) = mk$ and $O_m^*$ the pull-back given by $O_m^*(f)(k) = f(mk)$.
We have the following:
\begin{Th} \label{th-6.3}
If $p$ is sufficiently large and divisible by $m=m(A)$, the dimension
of the variety $Y_{L_p}$ is independent of $p$ and is equal to the dimension $q$
of the Newton-Okounkov body $\Delta(A)$. Let $\deg$
be the function given by $\deg(p)= \deg Y_{L_p}$.
Then the $q$-th growth coefficient of the function $O^*_m(\deg)$, i.e.
$$a_q(O^*_m(\deg)) = \lim_{k \to \infty} \frac{\deg Y_{L_{mk}}}{k^q},$$ exists and is equal to $q!a_q(O^*_m (H_A))$,
which in turn is equal to $q!\Vol_q(\Delta(A))/\ind(A)$.
\end{Th}
\begin{proof}
Follows from Theorem \ref{th-Fujita-algebra} and Hilbert's theorem.
\end{proof}


When $A$ is an algebra of integral type Theorem \ref{th-6.3} can be refined.
{Note that in this case $m(A)=1$.}
\begin{Th} \label{th-6.5}
Let $A$ be an algebra of integral type. Then for sufficiently large $p$, the
dimension $q$ of the variety $Y_{L_p}$, as well as the degree of the variety
$Y_{L_p}$ divided by $p^q$, are independent of $p$. Moreover,
the dimension of $Y_{L_p}$ is equal to the dimension of the Newton-Okounkov body $\Delta(A)$ and
its degree is given by:
$$\deg Y_{L_p} = q!p^q a_q(O^*_m (H_A)) = \frac{q!p^q\Vol_q(\Delta(A))}{\ind(A)}.$$
\end{Th}
\begin{proof}
Follows from Theorem \ref{th-5.16}, Hilbert's theorem and Theorem \ref{th-6.3}.
\end{proof}

\subsection{Algebras of almost integral type associated to linear series} \label{subsec-linear-series}
In this section we apply the results on graded algebras to the
rings of sections of divisors and more generally to linear series.
One of the main results is a generalization of the Fujita approximation theorem (for a big
divisor) to any divisor on a complete variety.

{Let $X$ be an irreducible variety of dimension $n$ over an algebraically closed field $\k$},
and let $D$ be a Cartier divisor on $X$. To $D$ one associates the subspace $\L(D)$ of
rational functions defined by
$$\L(D) = \{ f \in \k(X) \mid (f) + D \geq 0 \}.$$

Let $\mathcal{O}(D)$ denote the line bundle corresponding to $D$.
When $X$ is normal, the elements of
the subspace $\L(D)$ are in one-to-one correspondence with the sections
in $H^0(X, \mathcal{O}(D))$. One also knows that for a complete variety $X$ the dimension of $H^0(X, \mathcal{O}(D))$ is finite 
({see \cite[Chapter II, Theorem 5.19]{Hartshorne}}). Thus whenever $X$ is normal and complete, the vector space $\L(D)$ is
finite dimensional.

Let $D, E$ be divisors and let $f \in \L(D)$, $g \in \L(E)$. From the definition it is
clear that $fg \in \L(D+E)$. Thus multiplication of functions gives a map
\begin{equation} \label{equ-L(D+E)}
\L(D) \times \L(E) \to \L(D+E).
\end{equation}
In general {\it this map is not surjective}.

To a divisor $D$ we associate a graded subalgebra $\R(D)$ of the ring $F[t]$ of polynomials
in $t$ with coefficients in the field of rational functions $F=\c(X)$ as follows.
\begin{Def}
Define
$\R(D)$ to be the collection of all the polynomials $f(t) = \sum_k f_k t^k$ with $f_k \in \L(kD)$,
for all $k$. In other words,
$$\R(D) = \bigoplus_{k=0} \L(kD)t^k.$$
From (\ref{equ-L(D+E)}) it follows that $\R(D)$ is a graded subalgebra of $F[t]$.
\end{Def}

\begin{Rem}
One can find example of a divisor $D$ such that the algebra $\R(D)$ is not finitely generated.
See for example \cite[Section 2.3]{Lazarsfeld}.
\end{Rem}


\begin{Th} \label{th-ring-sec-almost-finite-type}
{For any Cartier divisor $D$ on a complete variety $X$}, the algebra $\R(D)$ is of almost integral type.
\end{Th}

To prove Theorem \ref{th-ring-sec-almost-finite-type} we need some preliminaries which we recall here. When $D$ is a very ample divisor, the following well-known result describes
$\R(D)$ (see \cite[Chapter II, Ex. 5.14]{Hartshorne}).
\begin{Th} \label{th-very-ample-int-closure}
Let $X$ be a normal projective variety and $D$ a very ample divisor. Let $L = \L(D)$ be the finite dimensional
subspace of rational functions associated to $D$, and let
$A_L = \bigoplus_{k \geq 0} L^kt^k$ be the algebra corresponding to $L$. Then
1) $\R(D)$ is the integral closure of $A_L$ in its field of fractions.
2) $\R(D)$ is a graded subalgebra of integral type.
\end{Th}

It is well-known that very ample divisors generate the group of all Cartier divisors
(see \cite[Example 1.2.10]{Lazarsfeld}). More precisely,
\begin{Th} \label{th-difference-very-ample}
Let $X$ be a projective variety. Let $D$ be a Cartier divisor and $E$ a very ample divisor. Then for large
enough $k$, the divisor $D+kE$ is very ample. In particular, $D$ can be written as the difference of two
very ample divisors $D+kE$ and $kE$.
\end{Th}

Finally we need the following statement which is an immediate corollary of
Chow's lemma and the normalization theorem.
\begin{Lem} \label{lem-Chow}
Let $X$ be any complete variety. Then
there exists a normal projective variety $X'$ and a morphism $\pi: X' \to X$ which is a
birational isomorphism.
\end{Lem}
\begin{proof}[Proof of Theorem \ref{th-ring-sec-almost-finite-type}]
Let $\pi: X' \to X$ be as in Lemma \ref{lem-Chow}.
Let $D' = \pi^*(D)$ be the pull-back of $D$ to $X'$. Then
$\pi^*(\R(D)) \subset \R(D')$. Thus replacing $X$ with $X'$, it is enough to
prove the statement when $X$ is normal and projective.
Now by Theorem \ref{th-difference-very-ample} we can find very ample divisors $D_1$ and $D_2$ with
$D = D_1 - D_2$, moreover, we can take $D_2$ to be an effective divisor.
It follows that $\R(D) \subset \R(D_1)$.
By Theorem \ref{th-very-ample-int-closure},
$\R(D_1)$ is of integral type and hence $\R(D)$ is of almost integral type.
\end{proof}

We can now apply the results of Section \ref{subsec-application-valuation}
to the graded algebra $\R(D)$ and derive some results
on the asymptotic of the dimensions of the spaces $\L(kD)$.

Let us recall some terminology from the theory of divisors and linear series
(see \cite[Chapter 2]{Lazarsfeld}). These are special cases of the
corresponding general definitions for graded algebras in Part II of the paper.

A graded subalgebra $W$ of $\R(D)$ is usually called a {\it graded linear series for $D$}.
Since $\R(D)$ is of almost integral type, then any graded linear series $W$ for $D$ is also an
algebra of almost integral type. Let us write $W = \bigoplus_{k \geq 0} W_k = \bigoplus_{k \geq 0} L_k t^k$,
where $W_k$ (respectively $L_k$) is the $k$-th homogeneous component (respectively $k$-th subspace)
of the graded subalgebra $W$.

1) The $n$-th growth coefficient of the algebra $W$ multiplied with $n!$,
is called the {\it volume} of the
graded linear series $W$ and denoted by $\Vol(W)$. When $W = \R(D)$, the volume
of $W$ is denoted by $\Vol(D)$. {In the classical case, i.e. when $D$ is ample, $\Vol(D)$ is
equal to its self-intersection number. In the case $\k = \c$ and $D$ very ample,
$\Vol(D)$ is the (symplectic) volume of the image of $X$ under the embedding of $X$ into projective space
induced by $D$, and hence the term volume.}

2) The index $m = m(W)$ of the algebra $W$ is usually called the {\it exponent} of the
graded linear series $W$. Recall that for large enough $p$ and divisible by $m$, the homogeneous
component $W_p$ is non-zero.

3) The growth degree $q$ of the Hilbert function of the algebra $W$
is called the {\it Kodaira-Iitaka dimension of $W$}.

The general theorems proved in Section \ref{subsec-application-valuation}
about algebras of almost integral type, applied to a graded
linear series $W$ gives the following results:

\begin{Cor} \label{cor-linear-series}
Let $X$ be a complete irreducible $n$-dimensional  variety. Let $D$ be a
Cartier divisor on $X$ and $W \subset \R(D)$ a graded linear series. Then:
1) The $q$-th growth coefficient of the function $O_m^*(H_W)$, i.e.
$$a_q(O^*_m(H_W)) = \lim_{k \to \infty} \frac{\dim W_{mk}}{k^q},$$ exists.
Fix a faithful $\z^n$-valued valuation for the field $\k(X)$. Then the Kodaira-Iitaka dimension $q$ of $W$ is
equal to the dimension of the convex body $\Delta(W)$ and
the growth coefficient $a_q(O^*_m(H_W))$ is equal to $\Vol_q(\Delta(W))$.
Following the notation for the volume of a divisor we denote the
quantity $q! a_q(O^*_m(H_W))$ by $\Vol_q(W)$.

2) (A generalized version of Fujita approximation)
For $p$ sufficiently large and divisible by $m$, let $\varphi(p)$ be the $p$-th
growth coefficient of the graded algebra $A_{L_p} =
\bigoplus_k L_p^k t^k$ associated to the $q$-th subspace $L_p$ of $W$, i.e.
$\varphi(p) = \lim_{t \to \infty}\dim L_p^t/t^q.$
Then the $q$-th growth coefficient of the function $O_m^*(\varphi)$, i.e.
$$a_q(O_m^*(\varphi)) = \lim_{k \to \infty} \frac{\varphi(mk)}{k^q},$$ exists and is equal to
$\Vol_q(\Delta(W))/\ind(W) = \Vol_q(W)/q!\ind(W)$.

3) (Brunn-Minkowski for volume of graded linear series)
Suppose $W_1$ and $W_2$ are two graded linear series for divisors $D_1$ and $D_2$ respectively. Also
assume $m(W_1) = m(W_2) = 1$, then we have: $$\Vol^{1/n}(W_1) + \Vol^{1/n}(W_2) \leq \Vol^{1/n}(W_1W_2),$$
where $W_1W_2$ denotes the componentwise product of $W_1$ and $W_2$. In particular,
if $W_1 = \R(D_1)$ and $W_2 = \R(D_2)$ then $W_1W_2 \subset \R(D_1+D_2)$ and hence
$$\Vol^{1/n}(D_1) + \Vol^{1/n}(D_2) \leq \Vol^{1/n}(D_1+D_2).$$
\end{Cor}

\begin{Rem} \label{rem-linear-series}
The existence of the limit in 1) has been known for the graded algebra $\R(D)$, where $D$ is
a so-called {\it big divisor} (see \cite{Lazarsfeld}).
A divisor $D$ is big if its volume $\Vol(D)$ is strictly positive.
Equivalently, $D$ is big if for some $k>0$, the Kodaira map of the subspace $\L(kD)$ is a birational isomorphism
onto its image.

It seems that for a general graded linear series (and in particular the algebra $\R(D)$ of
a general divisor $D$) the existence of the limit in 1) has not previously been known (see \cite[Remark 2.1.39]{Lazarsfeld}).

Part 2) above is in fact a generalization of the Fujita approximation result of
\cite[Theorem 3.3]{Lazarsfeld-Mustata}. Using similar methods, for certain graded linear series of big divisors,
Lazarsfeld and Mustata prove a statement very close to the statement 2) above.

In \cite{Lazarsfeld-Mustata} and
\cite[Theorem 5.13]{Askold-Kiumars-arXiv-1} the Brunn-Minkowski inequality in 3)
is proved with similar methods.
\end{Rem}
\section{Part IV: Applications to intersection theory and mixed volume}
In this part we associate a convex body to any non-zero finite dimensional
subspace of rational functions on an $n$-dimensional irreducible variety
such that: 1) the volume of the body multiplied by $n!$ is
equal to the self-intersection index of the subspace, 2) the body which
corresponds to the product of subspaces contains the sum of the
bodies corresponding to the factors. This construction allows us
to prove that the intersection index enjoys all the main inequalities concerning the mixed
volume, and also to prove these inequalities for the mixed volume itself.

\subsection{Mixed volume} \label{subsec-mixed-vol}
In this section we recall the notion of mixed volume of convex bodies and
list its main properties (without proofs).

{The collection of all convex bodies in $\r^n$ is a cone, that is, we can add convex bodies and multiply a convex body with a
positive number.} Let $\Vol$ denote the $n$-dimensional volume in $\r^n$ with respect to the
standard Euclidean metric. The function $\Vol$ is a homogeneous polynomial of degree $n$ on the cone of convex bodies,
i.e. its restriction to each finite dimensional section of the cone is a homogeneous polynomial of degree $n$.

By definition the {\it mixed volume}  $V(\Delta_1,\dots,\Delta_n)$ of
an $n$-tuple $(\Delta_1,\dots,\Delta_n)$ of convex bodies
is the coefficient of the monomial $\lambda_1\cdots\lambda_n$ in the polynomial
$$P_{\Delta_1,\dots,\Delta_n}(\lambda_1, \ldots, \lambda_n)
= \Vol(\lambda_1\Delta_1 + \cdots + \lambda_n\Delta_n),$$ divided by $n!$.
This definition implies that mixed volume is the {\it polarization} of the volume polynomial,
that is, {it is the unique function} on the $n$-tuples of convex bodies satisfying the
following:
\begin{itemize}
\item[(i)] (Symmetry) $V$ is symmetric with respect to permuting the bodies $\Delta_1, \ldots, \Delta_n$.
\item[(ii)] (Multi-linearity) It is {linear} in each argument with respect to the Minkowski sum. {The linearity} in first argument
means that for convex bodies $\Delta_1'$, $\Delta_1'',
\Delta_2,\dots,\Delta_n$, and real numbers $\lambda', \lambda'' \geq 0$ we have:
$$ V(\lambda'\Delta_1'+\lambda''\Delta_1'',\dots, \Delta_n)=\lambda'V(\Delta_1',\dots,
\Delta_n) + \lambda''V(\Delta_1'',\dots, \Delta_n).$$
\item[(iii)] (Relation with volume) On the diagonal it coincides with the volume, i.e. if
$\Delta_1 =\dots=\Delta_n=\Delta$, then $V(\Delta_1,\dots,
\Delta_n)=\Vol(\Delta)$.
\end{itemize}

It is easy to verify that: 1) Mixed volume is non-negative, and
2) Mixed volume is monotone, that is,
for two $n$-tuples of convex bodies $\Delta'_1\subset \Delta_1,\dots, \Delta'_n\subset
\Delta_n$ we have: $V(\Delta'_1,\dots, \Delta'_n) \leq V(\Delta_1,\dots, \Delta_n).$

The following inequality attributed to Alexandrov and Fenchel is important and very
useful in convex geometry. All its previously known proofs are rather complicated (see
\cite{Burago-Zalgaller}).
\begin{Th}[Alexandrov-Fenchel] \label{thm-alexandrov-fenchel}
Let $\Delta_1, \ldots, \Delta_n$ be convex bodies
in $\r^n$. Then
$$ V(\Delta_1, \Delta_1, \Delta_3, \ldots, \Delta_n)
V(\Delta_2, \Delta_2, \Delta_3, \ldots, \Delta_n) \leq V^2(\Delta_1, \Delta_2, \ldots, \Delta_n).$$
\end{Th}

In dimension $2$, this inequality is elementary. We will call it the {\it generalized isoperimetric inequality}, because when
$\Delta_2$ is the unit ball it coincides with the classical isoperimetric inequality.



The celebrated {\it Brunn-Minkowski inequality} concerns volume of
convex bodies in $\r^n$. {It is an easy corollary of the Alexandrov-Fenchel
inequality.}

\begin{Th}[Brunn-Minkowski] \label{th-Brunn-Mink}
Let $\Delta_1$, $\Delta_2$ be convex bodies in $\r^n$. Then
$$\Vol^{1/n}(\Delta_1) + \Vol^{1/n}(\Delta_2)\leq
\Vol^{1/n}(\Delta_1+\Delta_2).$$
\end{Th}
{We used this inequality in the proof of Theorem \ref{th-Brunn-Mink-algebra}.
In dimension $2$, the Brunn-Minkwoskii inequality is equivalent to the generalized isoperimetric inequaity
(compare with Corollary \ref{cor-Hodge}).

On the other hand, all the classical proofs of the Alexandrov-Fenchel inequality
deduce it from the Brunn-Minkowski inequality.
But these deductions are the main and
most complicated parts of the proofs (\cite{Burago-Zalgaller}).
Interestingly, the main construction in the present paper (using algebraic geometry)
allows us to obtain the Alexandrov-Fenchel inequality as an immediate corollary of its
simplest case, namely the generalized isoperimetric inequality (that is, when $n=2$).

\subsection{Semigroup of subspaces and intersection index} \label{subsec-int-index}
In this section we briefly review some concepts and results
from \cite{Askold-Kiumars-arXiv-2}. That is, we discuss the semigroup of subspaces of rational functions,
its Grothendieck group and the intersection index on the Grothendieck group. We also
recall the key notion of the completion of a subspace.

{For the rest of the paper we will take the ground field $\k$ to be the field of complex numbers $\c$.
Let $F$ be a field finitely generated over $\c$.
Later we will deal with the
case where $F=\c(X)$ is the field of rational functions on a variety $X$ over $\c$.}
Recall (Definition \ref{def-prod-subspace}) that ${\bf K}(F)$ denotes
the collection of all non-zero finite dimensional subspaces of $F$ over $\c$.
Moreover, for $L_1, L_2 \in {\bf K}(F)$, the product
$L_1L_2$ is the $\k$-linear subspace spanned by all the products $fg$ where $f \in L_1$ and $g \in L_2$.
With respect to this product ${\bf K}(F)$ is a (commutative) semigroup.

In general the semigroup ${\bf K}(F)$ does not have the cancellation property.
that is, the equality $L_1M = L_2M$, $L_1, L_2, M \in {\bf K}(F)$, does not imply $L_1 = L_2$.
Let us say that $L_1$ and $L_2$ are {\it equivalent}
and write $L_1 \sim L_2$, if there is $M \in {\bf K}(F)$ with $L_1M = L_2M$.
Naturally the quotient ${\bf K}(F) / \sim$ is a semigroup with the cancellation property and
hence can be extended to a group. The {\it Grothendieck group}
${\bf G}(F)$ of ${\bf K}(F)$ is the collection of formal quotients $L_1/L_2$, $L_1, L_2 \in {\bf K}(F)$, where
$L_1/L_2 = L_1'/L_2'$ if $L_1L_2' \sim L_1'L_2$. There is a natural homomorphism
$\phi: {\bf K}(F) \to {\bf G}(F)$.
The Grothendieck group has the following universal property: for any group $G'$ and a
homomorphism $\phi': {\bf K}(F) \to G'$, there exists a unique homomorphism
$\psi: {\bf G}(F) \to G'$  such that $\phi' = \psi \circ \phi$.

Similar to the notion of integrality of an element over a ring, one defines the integrality of
an element over a linear subspace.

\begin{Def} \label{def-int-over-L}
Let $L$ be a ${\bf k}$-linear subspace in $F$.
An element $f \in F$ is {\it integral over} $L$ if
it satisfies an equation
\begin{equation} \label{equ-int-element-L}
f^m + a_1f^{m-1} + \cdots + a_m = 0,
\end{equation}
where $m>0$ and $a_i \in L^i$, $i=1, \ldots, m$.
The {\it completion} or {\it integral closure} $\ol{L}$ of $L$ in
$F$ is the collection of all $f \in F$ which are integral over $L$.
\end{Def}

The facts below about the completion of a subspace can be found, for example, in \cite[Appendix 4]{Zariski}.
One shows that $f \in F$ is integral over a subspace $L$ if and only if
$L \sim L + \langle f \rangle$. Moreover, the completion $\ol{L}$ is
a subspace containing $L$, and if $L$ is finite dimensional then
$\ol{L}$ is also finite dimensional.

The completion $\ol{L}$ of a subspace $L \in {\bf K}(F)$ can be characterized
in terms of the notion of equivalence of subspaces: take $L \in {\bf K}(F)$. Then $\ol{L}$ is
the largest subspace which is equivalent to $L$, that is,
1) $L \sim \ol{L}$ and, 2) If $L \sim M$ then $M \subset \ol{L}$.

The following standard result shows the connection between the completion of subspaces
and integral closure of algebras.
\begin{Th} \label{th-int-closure-A_L}
Let $L$ be a finite dimensional ${\bf k}$-subspace and let $A_L = \bigoplus_k L^k t^k$ be the
corresponding graded subalgebra of $F[t]$. Then the $k$-th subspace of the integral closure $\ol{A_L}$ is
$\ol{L^k}$, the completion of the $k$-th subspace of $A_L$. That is, $$\ol{A_L} = \bigoplus_k \ol{L^k}t^k.$$
\end{Th}

Consider an $n$-dimensional irreducible algebraic variety $X$
and let $F = \c(X)$ be its field of rational functions.
We denote the semigroup ${\bf K}(F)$ of finite dimensional subspaces in $F$ by ${\bf K}_{rat}(X)$.

Next we recall the notion of intersection index of an $n$-tuple of subspaces.
Let $(L_1, \ldots, L_n)$ be an $n$-tuple of finite dimensional spaces of rational functions on $X$.
Put ${\bf L} = L_1 \times \cdots \times L_n$.
Let $U_{\bf L} \subset X$ be the set of all non-singular points in $X$ at which all the
functions from all the $L_i$ are regular, and let $Z_{\bf L} \subset U_{\bf L}$ be the collection of
all the points $x$ in $U_{\bf L}$ such that all the functions from some $L_i$ ($i$ depending on $x$) vanish at $x$.

\begin{Def}
Let us say that for an $n$-tuple of subspaces $(L_1, \ldots, L_n)$, the intersection index
is defined and equal to $[L_1, \ldots, L_n]$ if there is a proper algebraic subvariety ${\bf R}\subset {\bf L}=L_1 \times \cdots \times L_n$ such that for each $n$-tuple $(f_1,\dots,f_n)\in {\bf L} \setminus {\bf R}$
the following holds:

1) The number of solutions of the system $f_1=\dots=f_n=0$ in the set
$U_{\bf L} \setminus Z_{\bf L}$ is independent of the choice of $(f_1, \ldots, f_n)$ and is equal to
$[L_1,\dots,L_n]$.

2) Each solution $a\in U_{\bf L} \setminus Z_{\bf L}$ of the system $f_1=\dots=f_n=0$
is non-degenerate, i.e. the form $df_1\wedge\dots\wedge
df_n$ does not vanish at $a$.
\end{Def}

The following is proved in \cite[Proposition 5.7]{Askold-Kiumars-arXiv-2}.
\begin{Th}
For any $n$-tuple $(L_1, \ldots, L_n)$ of subspaces $L_i \in {\bf K}_{rat}(X)$ the
intersection index $[L_1 \ldots, L_n]$ is defined.
\end{Th}

The following are immediate corollaries of the definition of the intersection index:
1) $[L_1,\dots,L_n]$ is a symmetric function of
$L_1,\dots,L_n \in {\bf K}_{rat}(X)$, 2) The intersection index
is monotone, (i.e. if $L'_1\subseteq L_1,\dots, L'_n\subseteq L_n$,
then $[L'_1,\dots,L'_n] \leq [L_1,\dots,L_n]$, and 3) The intersection index is non-negative.

The next theorem contains the main properties of the intersection index
(see \cite[Section 5]{Askold-Kiumars-arXiv-2}).
\begin{Th} \label{th-multi-lin-int-index}
1) (Multi-additivity) Let $L_1', L_1'', L_2, \ldots, L_n \in {\bf K}_{rat}(X)$
and put $L_1= L_1'L_1''$. Then
$$[L_1,\dots,L_n]=[L'_1,L_2,\dots,L_n]+[L''_1,L_2,\dots,L_n].$$
2) (Invariance under the completion)
Let $L_1 \in {\bf K}_{rat}(X)$ and let $\overline{L_1}$ be its completion.
Then for any $(n-1)$-tuple $L_2,\dots,L_n\in {\bf K}_{rat}(X)$ we have:
$$[L_1,L_2,\dots,L_n]=[\overline{L_1},L_2,\dots,L_n].$$
\end{Th}

Because of the multi-additivity, the intersection index can be extended to the
Grothendieck group ${\bf G}_{rat}(X)$ of the semigroup $\K$.
The Grothendieck group of ${\bf K}_{rat}(X)$ can be considered as
an analogue (for a typically non-complete variety $X$) of the group of Cartier
divisors on a complete variety, and the intersection index on
the Grothendieck group ${\bf G}_{rat}(X)$ as an analogue of the intersection
index of Cartier divisors.

The next proposition relates the self-intersection index of a subspace with the degree of
the image of the Kodaira map. It easily follows from the definition of the intersection index.
\begin{Prop}[Self-intersection index and degree] \label{prop-self-int-deg}
Let $L \in \K$ be a subspace and $\Phi_L: X \ratmap Y_L \subset \p(L^*)$ its
Kodaira map. 1) If $\dim X = \dim Y_L$ then $\Phi_L$ has finite mapping degree $d$ and
$[L, \ldots, L]$ is equal to the degree of the subvariety $Y_L$ (in $\p(L^*)$) multiplied with $d$.
2) If $\dim X > \dim Y_L$ then $[L, \ldots, L] = 0$.
\end{Prop}

\subsection{Newton-Okounkov body and intersection index} \label{sec-main}
We now discuss the relation between the self-intersection index of a subspace of rational functions
and the volume of the Newton-Okounkov body.

Let $X$ be an irreducible $n$-dimensional variety (over $\c$) and
$L \in \K$ a non-zero finite dimensional subspace of rational functions.
We can naturally associate two algebras of integral type to $L$: the
algebra $A_L$ and its integral closure $\ol{A_L}$. (Note that by Theorem \ref{th-A-bar-finite-type},
$\ol{A_L}$ is an algebra of integral type.)

Let $F = \c(X)$. As in Section \ref{subsec-application-valuation},
let $v:F \setminus \{0\} \to \z^n$ be a faithful valuation with one-dimensional leaves,
and $v_t$ its extension to the
polynomial ring $F[t]$. Then $v_t$ associates two convex bodies to the space $L$,
namely $\Delta(A_L)$ and $\Delta(\ol{A_L})$.

Since $A_L \subset \ol{A_L}$ then $\Delta(A_L) \subset \Delta(\ol{A_L})$.
In general $\Delta(\ol{A_L})$ can be strictly bigger than $\Delta(A_L)$ (see Example \ref{ex-A_L-A_L-bar}).

The following two theorems can be considered as far generalizations of the Kushnirenko theorem
in Newton polytope theory and toric geometry. Below $\Vol_n$ denotes the standard Euclidean
measure in $\r^n$.
\begin{Th} \label{th-main-Delta-A_L}
Let $L \in \K$ with the Kodaira map $\Phi_L$.
1) If $\Phi_L$ has finite mapping degree then
$$[L, \ldots, L] = \frac{n! \deg \Phi_L}{\ind(A_L)} \Vol_n(\Delta(A_L)).$$
Otherwise, both $[L, \ldots, L]$ and $\Vol_n(\Delta(A_L))$ are equal to $0$.
2) In particular, if $\Phi_L$ is a birational isomorphism between $X$ and $Y_L$ then
$\deg \Phi_L = \ind(A_L) = 1$ and we obtain
$$[L, \ldots, L] = n! \Vol_n(\Delta(A_L)).$$
3) The correspondence $L \mapsto \Delta(A_L)$ is superadditive, i.e.
if $L_1, L_2$ are finite dimensional subspaces of rational functions. Then
$$\Delta(A_{L_1}) \oplus_t \Delta(A_{L_2}) \subset \Delta(A_{L_1L_2}).$$
(In other words, $\Delta_0(A_{L_1}) + \Delta_0(A_{L_2}) \subset \Delta_0(A_{L_1L_2}),$
where $\Delta_0$ is the Newton-Okounkov body projected to the level $0$ and $+$ is the
Minkowski sum.)
\end{Th}
\begin{proof}
1) Follows from Proposition \ref{prop-self-int-deg} and Corollary \ref{cor-deg-Y_L-vol-Delta}.
2) If $\Phi_L$ is a birational isomorphism then it has degree $1$.
On the other hand, from the birational isomorphism of $\Phi_L$ it follows
that the subfield $P(L)$ associated to $L$ coincides with the whole field $\c(X)$.
Since the valuation $v$ is faithful we then conclude that the subgroup $G_0(A_L)$ coincides
with the whole $\z^n$ and hence $\ind(A_L) = 1$. Part 2) then follows from 1).
3) We know that $A_{L_1L_2} = A_{L_1}A_{L_2}$ and $m(A_{L_1}) = m(A_{L_2}) = m(A_{L_1L_2}) = 1$.
Proposition \ref{prop-superadd-alg} now gives the required result.
\end{proof}

\begin{Th} \label{th-main-Delta-bar-A_L}
1) We have:
$$[L, \ldots, L] = n!\Vol_n(\Delta(\ol{A_L)}).$$
2) The correspondence $L \mapsto \Delta(\ol{A_L})$ is superadditive, i.e.
if $L_1, L_2$ are finite dimensional subspaces of rational functions. Then
$$\Delta(\ol{A_{L_1}}) \oplus_t \Delta(\ol{A_{L_2}}) \subset \Delta(\ol{A_{L_1L_2}}).$$
(In other words, $\Delta_0(\ol{A_{L_1}}) + \Delta_0(\ol{A_{L_2}}) \subset \Delta_0(\ol{A_{L_1L_2}})$.)
\end{Th}
We need the following lemma.
\begin{Lem} \label{lem-Lp-bar}
Let $L$ be a subspace of rational functions. Suppose $\dim Y_L = n$ i.e. the Kodaira map $\Phi_L$
has finite mapping degree. Then there exists $N>0$ such that the following holds:
for any $p>N$ the subfield associated to the completion $\ol{L^p}$ coincides with the whole $\c(X)$.
\end{Lem}
\begin{proof}
Let $E =  P(L) \cong \c(Y_L)$ and $F = \c(X)$.
The extension $F/E$ is a finite extension because $\dim Y_L=n$.
Clearly for any $p>0$, $P(\ol{L^p}) \subset F$. We will show that
there is $N>0$ such that for $p>N$ we have $F \subset P(\ol{L^p})$.
Let $f_1, \ldots, f_r$ be a basis for $F/E$. Let $f \in \{f_1, \ldots
f_r\}$. Then $f$ satisfies an equation
\begin{equation} \label{equ-lem-Lp-bar}
a_0f^m + \cdots + a_m = 0,
\end{equation}
where $a_i = P_i/Q_i$
with $P_i, Q_i \in L^{d_i}$ for some $d_i > 0$. Let $N_f = \sum_{i=0}^m d_i$
and put
$Q = Q_0 \cdots Q_m$. Then $Q \in L^{N_f}$. Multiplying
(\ref{equ-lem-Lp-bar}) with $Q$
we have $b_0 f^m + \cdots + b_m = 0,$ where $b_i = P_iQ/Q_i \in L^{N_f}$.
Then multiplying with $b_0^{m-1}$
gives $$(b_0f)^m + b_{1}(b_0f)^{m-1}+ \cdots + (b_0^{m-1})b_m = 0,$$ which
shows that
$b_0f$ is integral over $L^{N_f}$. Now
$f \in P(\ol{L^{N_f}})$ because $f = b_0f/b_0$ and $b_0 \in L^{N_f}$.
Let $N$ be the maximum of the $N_f$ for $f \in \{f_1, \ldots, f_r\}$.
It follows that
$F \subset P(\ol{L^N})$. It is easy to see that for $p>N$ we have
$\ol{L^N}L^{p-N} \subset \ol{L^p}$ and hence $P(\ol{L^N}) \subset P(\ol{L^p})$. Thus
$F = P(\ol{L^p})$ as required.
\end{proof}

\begin{proof}[Proof of Theorem \ref{th-main-Delta-bar-A_L}]
1) Suppose $\dim Y_L < n$. Then from the definition of the self-intersection index it follows that
$[L, \ldots, L] = 0$. But we know that the $n$-dimensional volume of $\Delta(\ol{A_L})$ is $0$ because $\dim \Delta(\ol{A_L})$ equals the dimension of $Y_L$ and hence is less than $n$. This proves the theorem
in this case. Now suppose $\dim Y_L = n$. Then the Kodaira map $\Phi_L$ has finite mapping degree.
By Lemma \ref{lem-Lp-bar}, the Kodaira map $\Phi_{\ol{L^p}}$ is a
birational isomorphism onto its image. Thus the self-intersection index of
$\ol{L^p}$ is equal to the degree of the variety $Y_{\ol{L^p}}$.
By the main properties of intersection index (Theorem \ref{th-multi-lin-int-index}) we have:
$$[\ol{L^p}, \ldots, \ol{L^p}] = [L^p, \ldots, L^p] = p^n[L, \ldots, L].$$
On the other hand, by Theorem \ref{th-5.16},
$$\deg Y_{\ol{L^p}} = \frac{p^n}{\ind(\ol{A_L})} \Vol_n(\Delta(\ol{A_L})).$$
But since the field $P(\ol{L^p})$ coincides with $\c(X)$ we have:
$\ind(\ol{A_L}) = 1$ which finishes the proof of 1).
To prove 2) first note that we have the inclusion:
$$\ol{A_{L_1}}~\ol{A_{L_2}} \subset \ol{A_{L_1L_2}},$$
(this follows from the fact that for any two subspaces $L, M$ we have
$\ol{L}~\ol{M} \subset \ol{LM}$).
Secondly, since $m(\ol{A_{L_1}}) = m(\ol{A_{L_2}}) = 1$ by Proposition \ref{prop-superadd-alg} we know:
$$\Delta(\ol{A_{L_1}}) \oplus_t \Delta(\ol{A_{L_2}}) \subset \Delta(\ol{A_{L_1}}~\ol{A_{L_2}}) \subset \Delta(\ol{A_{L_1L_2}}).$$ The theorem is proved.
\end{proof}

\begin{Ex} \label{ex-A_L-A_L-bar}
{Let $X$ be the affine line $\c$ and  let $z$ denote the coordinate function on it}. Let $L =
\span\{1, z^2\}$. Clearly $[L] = 2$.
Now let $v: \c(X) \to \z$ be the order of vanishing at the point $1 \in X$.
Then $S(A_L) = \z_{\geq 0}$ and hence $\ind(A_L)=1$. This shows that $\deg(\Phi_L) \gneqq \ind(A_L)$.
Also it is easy to see that $\Delta(A_L)$ is the line segment $[0,1]$.
On the other hand one sees that $\Delta(\ol{A_L})$ is the line segment $[0,2]$.
\end{Ex}


Next, let us see that the well-known Bernstein-Kushnirenko theorem follows from Theorem \ref{th-main-Delta-A_L}.
For this, we take the variety $X$ to be $(\c^*)^n$ and the
subspace $L$ a subspace spanned by Laurent monomials.

We identify the lattice $\z^n$ with the {\it Laurent monomials} in
$(\c^*)^n$: to each integral point $a=(a_1, \ldots, a_n) \in \z^n$,
we associate the monomial $x^a=x_1^{a_1}\dots
x_n^{a_n}$ where $x=(x_1, \ldots, x_n)$.
A {\it Laurent polynomial} $P(x) =\sum_a c_a x^a$ is a finite
linear combination of Laurent monomials with complex coefficients.
The {\it support} $\supp(P)$ of a Laurent polynomial $P$, is the set
of exponents $a$ for which $c_a \neq 0$. We denote the convex hull
of a finite set $I \subset \z^n$ by $\Delta_I \subset \r^n$.
The {\it Newton polytope} $\Delta (P)$ of a Laurent polynomial $P$
is the convex hull $\Delta_{\supp(P)}$ of its support. With each
finite set $I \subset \z^n$ one associates the linear space
$L(I)$ of Laurent polynomials $P$ with $\supp(P) \subset I$.


\begin{Th}[Kushnirenko] \label{th-Kush}
The number of solutions in $(\c^*)^n $ of a {general} system of Laurent polynomial
equations $P_1=\dots=P_n=0$, with $P_1, \ldots, P_n \in L(I)$ is equal to
$n! \Vol(\Delta_I)$, i.e. $$[L(I), \dots, L(I)]=n! \Vol(\Delta_I).$$
\end{Th}
\begin{proof}
If $I$, $J$ are finite subsets in $\z^n$ then $L_IL_J = L_{I+J}$. Consider the graded algebra
$A_{L(I)} \subset F[t]$ where $F$ is the field of rational functions on $(\c^*)^n$.
Take any valuation $v$ on $F$ coming from the Gr\"{o}bner valuation on the field of fractions of the
algebra of formal power series $\c[[x_1, \ldots, x_n]]$ (see the paragraph before Example \ref{ex-valuation-rational-function-alg-var}). From the definition it is easy to see that
$v(L(I) \setminus \{0\}) = I$ and more generally $v(L(I)^k \setminus \{0\}) = k*I$, where
$k*I$ is the sum of $k$ copies of the set $I$. Let $S = S(A_{L(I)})$ be the semigroup
associated to the algebra $A_{L(I)}$. Then the cone $\Con(S)$ is the cone in $\r^{n+1}$ over
$\Delta_I \times \{1\}$ and the group $G_0(S)$ is generated by the differences $a-b$,
$a,b \in I$. Let $I = \{\alpha_1, \ldots, \alpha_r\}$. Then $\{x^{\alpha_1}, \ldots, x^{\alpha_r}\}$ is a basis
for $L(I)$ and one verifies that, in the dual basis for $L(I)^*$, the Kodaira map is given by
$\Phi_{L(I)}(x) = (x^{\alpha_1}: \cdots : x^{\alpha_r})$. From this it follows that the mapping
degree of $\Phi_{L(I)}$ is equal to the index of the subgroup $G_0(S)$, i.e. $\ind(A_{L(I)})$.
By Theorem \ref{th-main-Delta-A_L} we then have:
$$[L(I), \ldots, L(I)] = n! \frac{\deg \Phi_{L(I)}}{\ind(A_{L(I)})}
\Vol(\Delta_I) = n! \Vol(\Delta_I),$$ which proves the theorem.
\end{proof}

The Bernstein theorem computes the intersection index of
an $n$-tuple of subspaces of Laurent polynomials in terms of the mixed volume of
their Newton polytopes.
\begin{Th}[Bernstein] \label{th-Bernstein}
Let $I_1, \ldots, I_n \subset \z^n$ be finite subsets.
The number of solutions in $(\c^*)^n $ of a {general} system of Laurent polynomial
equations $P_1=\dots=P_n=0$, where $P_i \in L(I_i)$, is equal to
$n!V(\Delta_{I_1}, \ldots, \Delta_{I_n})$, i.e.
$$[L(I_1), \dots, L(I_n)]=n! V(\Delta_{I_1}, \ldots, \Delta_{I_n}),$$
(where, as before, $V$ denotes the mixed volume of convex bodies in $\r^n$).
\end{Th}
\begin{proof}
The Bernstein theorem readily follows from the multi-additivity of the intersection index, Theorem \ref{th-Kush}
(the Kushnirenko theorem) and the observation that for any two finite subsets
$I, J \subset \z^n$ we have $L(I+J) = L(I)L(J)$ and $\Delta_I + \Delta_J = \Delta_{I+J}$.
\end{proof}
The above proofs of the Bernstein and Kushnirenko theorems are in fact very close to
the ones in \cite{Askold-Hilbert-poly}.

\subsection{Proof of the Alexandrov-Fenchel inequality and its algebraic analogue}
\label{sec-Alexandrov-Fenchel}
Finally in this section, using the notion of Newton-Okounkov body,
we prove an algebraic analogue of the Alexandrov-Fenchel
inequality. From this we deduce the classical Alexandrov-Fenchel
inequality in convex geometry.

As be before $X$ is an $n$-dimensional irreducible complex algebraic variety.
The self-intersection index enjoys the following analogue of the Brunn-Minkowski inequality.
\begin{Cor} \label{cor-Brunn-Mink-int-index}
Let $L_1, L_2$ be finite dimensional subspaces in $\c(X)$ and $L_3 = L_1L_2$.
Then
$$[L_1, \ldots L_1]^{1/n} + [L_2, \ldots, L_2]^{1/n} \leq [L_3, \ldots, L_3]^{1/n}.$$
\end{Cor}
\begin{proof}
By definition the Newton-Okounkov body $\Delta(A)$ of an algebra $A$ lives at level $1$.
For $i=1,2,3$, let $\Delta_i$ be the Newton-Okounkov body of the algebra $\ol{A_{L_i}}$
projected to the level $0$. Then by Theorem \ref{th-main-Delta-bar-A_L}
we have $[L_i, \ldots, L_i] = n! \Vol_n(\Delta_i)$,
and $\Delta_1 + \Delta_2 \subset \Delta_3$. Now by the classical Brunn-Minkowski
inequality $\Vol^{1/n}_n(\Delta_1) + \Vol^{1/n}_n(\Delta_2) \leq \Vol^{1/n}_n(\Delta_3)$ which
proves the corollary (see also Theorem \ref{th-Brunn-Mink-algebra}).
\end{proof}

Surprisingly the most important case of the above inequality is the $n=2$ case, i.e. when
$X$ is an algebraic surface. As we show next, the general case of the above inequality and many other
inequalities for the intersection index follow from this $n=2$ case and the basic
properties of the intersection index.

\begin{Cor}[A version of the Hodge inequality] \label{cor-Hodge}
Let $X$ be an irreducible algebraic surface and let $L,M$ be non-zero finite dimensional subspaces of
$\c(X)$. Then
$$[L,L][M,M]\leq [L,M]^2.$$
\end{Cor}
\begin{proof}
From Corollary \ref{cor-Brunn-Mink-int-index}, for $n=2$, we have:
\begin{eqnarray*}
[L,L] +2[L,M] +[M,M] &=& [LM,LM] \cr &\geq& ([L,L]^{1/2}
+[M,M]^{1/2})^2  \cr &\geq& [L,L]+2[L,L]^{1/2}[M,M]^{1/2}+[M,M], \cr
\end{eqnarray*}
which readily implies the claim.
\end{proof}

In other words, Theorem \ref{th-main-Delta-bar-A_L} allowed us to easily reduce the Hodge
inequality above to the generalized isoperimetric inequality. We can now give an easy
proof of the Alexandrov-Fenchel inequality
for the intersection index.

Let us call a subspace $L \in \K$ a {\it very big subspace} if the Kodaira rational map
of $L$ is a birational isomorphism between $X$ and its image. Also we call a subspace
{\it big} if for some $m>0$, the subspace $L^m$ is very big. It is not hard to show that
the product of two big subspaces is again a big subspace and thus the big subspaces form a subsemigroup
of $\K$.

\begin{Th}[A version of the Bertini-Lefschetz theorem] \label{th-Lefschetz}
Let $X$ be a smooth irreducible $n$-dimensional variety
and let  $L_1,\dots,L_k\in \K$, $k<n$, be very big subspaces.
Then there is a Zariski open set ${\bf U}$ in ${\bf
L}=L_1\times\dots\times L_k$ such that for each point  ${\bf
f}= (f_1,\dots,f_k)\in {\bf U}$ the variety $X_{\bf f}$ defined in $X$ by
the system of equations $f_1=\dots=f_k=0$ is smooth and
irreducible.
\end{Th}
A proof of the Bertini-Lefschetz theorem can be found in \cite[Chapter II, Theorem 8.18]{Hartshorne}

One can slightly extend Theorem \ref{th-Lefschetz}. Assume that we are given $k$ very big
spaces $L_1, \ldots, L_k \in \K$ and $(n-k)$ arbitrary subspaces $L_{k+1}, \ldots, L_n$.
We denote by $[L_{k+1}, \ldots, L_n]_{X_{\bf f}}$, the intersection index of the restriction of the
subspaces $L_{k+1}, \ldots, L_n$ to the subvariety $X_{\bf f}$.
It is easy to verify the following reduction theorem.
\begin{Th} \label{th-reduction}
There is a Zariski open subset
${\bf U}$ in $L_1 \times \cdots \times L_k$ such that for ${\bf f}
=(f_1,\dots,f_k)\in {\bf U}$, the system $f_1=\dots=f_k=0$ defines a smooth irreducible
subvariety $X_{{\bf f}}$ in $X$ and the identity
$$[L_1,\dots L_n]_X =  [L_{k+1},\dots L_n]_{X_{\bf f}},$$
holds.
\end{Th}

\begin{Th}[Algebraic analogue of the Alexandrov-Fenchel
inequality] \label{th-Alex-Fenchel-alg} Let $X$ be an irreducible
$n$-dimensional variety and let $L_1,\dots,L_n\in \K$. Also assume that
$L_3, \ldots, L_n$ are big subspaces.
Then the following inequality holds:
$$[L_1,L_1,L_3,\dots,L_n][L_2,L_2,L_3,\dots,L_n] \leq [L_1,L_2, L_3,\dots,L_n]^2.$$
\end{Th}
\begin{proof}
Because of the multi-additivity of the intersection index, if the inequality holds for the spaces $L_i$ replaced
with ${L_i}^N$, for some $N$, then it holds
for the original spaces $L_i$. So without loss of generality we can assume that $L_3, \ldots, L_n$ are
very big. By Theorem \ref{th-reduction}, for almost all the $(f_3, \ldots, f_n) \in
L_3 \times \cdots \times L_n$ and the variety $Y$ defined by the system
$f_3 = \cdots = f_n = 0$ we have:
$$[L_1,L_2, L_3,\dots,L_n] = [L_1, L_2]_Y,$$
$$[L_1,L_1, L_3,\dots,L_n] = [L_1, L_1]_Y,$$
$$[L_2,L_2, L_3,\dots,L_n] = [L_2, L_2]_Y.$$
Now applying Corollary \ref{cor-Hodge} (the Hodge inequality) for the surface $Y$ we have:
$$[L_1, L_1]_Y[L_2, L_2]_Y \leq [L_1, L_2]^2_Y,$$ which proves the theorem.
\end{proof}

Let us introduce a notation for repetition
of subspaces in the intersection index.
Let $2\leq m\leq n$ be an integer and $k_1+\dots+k_r=m$ a partition of $m$ with $k_i \in \n$.
Consider the subspaces $L_1, \ldots, L_n \in \K$.
Denote by $[k_1*L_1,\dots, k_r*L_r,L_{m+1},\dots,L_n]$
the intersection index of $L_1, \ldots, L_n$ where
$L_1$ is repeated $k_1$ times, $L_2$ is repeated $k_2$ times, etc.
and $L_{m+1},\dots, L_n$ appear once.

\begin{Cor}[Corollaries of the algebraic analogue of
the Alexandrov-Fenchel inequality] \label{cor-Alex-Fenchel-alg} Let $X$
be an $n$-dimensional irreducible variety.
1) Let $2\leq m\leq n$ and $k_1+\dots+k_r=m$ with $k_i \in
\n$. Take big subspaces of rational functions $L_1, \ldots, L_n
\in \K$. Then $$\prod_{1 \leq j \leq r}[m*L_j, L_{m+1}, \ldots,
L_n]^{k_j} \leq [k_1*L_1, \ldots, k_r*L_r, L_{m+1},
\ldots, L_n]^m.$$
2)(Generalized Brunn-Minkowski inequality) For any fixed
big subspaces $L_{m+1}, \ldots, L_n \in \K$, the
function $$F: L \mapsto [m*L, L_{m+1}, \ldots, L_n]^{1/m},$$ is a
concave function on the semigroup $\K$.
\end{Cor}

1) follows formally from the algebraic analogue
of the Alexandrov-Fenchel, the same way that the corresponding
inequalities follow from the classical Alexandrov-Fenchel in convex geometry.
2) can be easily deduced from Corollary \ref{cor-Brunn-Mink-int-index} and
Theorem \ref{th-reduction}.

We now prove the classical Alexandrov-Fenchel inequality in convex geometry
(Theorem \ref{thm-alexandrov-fenchel}).

\begin{proof}[Proof of Theorem \ref{thm-alexandrov-fenchel}]
As we saw above, the Bernstein-Kushnirenko theorem follows from
Theorem \ref{th-main-Delta-A_L}. Applying the algebraic analogue of the Alexandrov-Fenchel
inequality to the situation considered in the Bernstein-Kushnirenko theorem
one proves the Alexandrov-Fenchel inequality for convex
polytopes of full dimension and with integral vertices. The homogeneity then implies
the Alexandrov-Fenchel inequality for convex polytopes of full dimension and with rational
vertices. But since any convex body can be approximated by convex polytopes of full dimension and with
rational vertices, by continuity we obtain
the Alexandrov-Fenchel inequality in complete generality.
\end{proof}


\vspace{.2cm}
{\Small
A. G. Khovanskii: Department of Mathematics, University of Toronto, Toronto,
Canada; Moscow Independent University; Institute for Systems Analysis, Russian Academy of Sciences.
email: {\sf askold@math.toronto.edu}

Kiumars Kaveh: Department of Mathematics, University of Pittsburgh,
Pittsburgh, PA, USA. email:\\ {\sf kaveh@pitt.edu}
}



\begin{thebibliography}{99}
\bibitem[Arnold98]{Arnold} Arnold, V. I. {\it
Higher dimensional continued fractions}. Regular and Chaotic
Dynamics, v.3(3), pp. 10-17, (1998).


\bibitem[Bernstein75]{Bernstein} Bernstein, D. N.
{\it The number of roots of a system of equations}.
English translation: Functional Anal. Appl. 9 (1975), no. 3, 183--185 (1976).

\bibitem[Boucksom-Chen09]{BC}
Boucksom, S.; Chen, H.
{\it Okounkov bodies of filtered linear series}. {arXiv:0911.2923}

\bibitem[Burago-Zalgaller88]{Burago-Zalgaller} Burago, Yu. D.; Zalgaller, V. A.
{\it Geometric inequalities}. Translated from the Russian by A. B. Sosinskii.
Grundlehren der Mathematischen Wissenschaften, 285.
Springer Series in Soviet Mathematics (1988).

\bibitem[Chulkov-Khovanskii06]{Askold-Chulkov} Chulkov, S.; Khovanskii, A. G.
{\it Geometry of semigroup $\z^n_{\geq 0}$ and its appearances in combinatroics,
algebra and differential equations}. MCCME, Moscow, 2006, 128 pp.

\bibitem[Eisenbud95]{Eisenbud} Eisenbud, D.
{\it Commutative algebra. With a view toward algebraic geometry}.
Graduate Texts in Mathematics, 150. Springer-Verlag, New York, 1995.

\bibitem[Fujita94]{Fujita} Fujita, T.
{\it Approximating Zariski decomposition of big line bundles}.
Kodai Math. J. 17 (1994), no. 1, 1--3.

\bibitem[Hartshorne77]{Hartshorne} Hartshorne, R. {\it Algebraic geometry}.
Graduate Texts in Mathematics, No. 52. Springer-Verlag, New York-Heidelberg, 1977.

\bibitem[Jacobson80]{Jacobson}
Jacobson, N. {\it Basic algebra. II}. W. H. Freeman and Company, New York, 1980.

\bibitem[Jow10]{Jow} Jow, Sh. {\it Okounkov bodies and restricted volumes along very general curves}.
{Adv. Math.} 223 (2010), no. 4, 1356Ð-1371.

\bibitem[Kaveh-Khovanskii08a]{Askold-Kiumars-arXiv-1} Kaveh, K.; Khovanskii, A. G.
{\it Convex bodies and algebraic equations on affine varieties}. Preprint: {arXiv:0804.4095v1}.
A short version with title {\it Algebraic equations and convex bodies}
to appear in {\it Perspectives in Analysis, Topology and Geometry},
Birkh\"aser series {Progress in Mathematics}.

\bibitem[Kaveh-Khovanskii08b]{Askold-Kiumars-arXiv-2}
Kaveh, K.; Khovanskii, A. G. {\it Mixed volume and an extension of intersection theory of
divisors}. {Moscow Mathematical Journal} 10 (2010), no. 2, 343--375.

\bibitem[Kaveh-Khovanskii10a]{Askold-Kiumars-reductive}
Kaveh, K.; Khovanskii, A. G. {\it Convex bodies associated to actions of reductive groups}.
{arXiv:1001.4830}.

\bibitem[Kaveh-Khovanskii10b]{Askold-Kiumars-Kaz}
Kaveh, K.; Khovanskii, A. G. {\it Moment polytopes, semigroup of representations and Kazarnovskii's theorem}.
{Journal of fixed point theory and applications} Vol. 7, Number 2, 401-417.

\bibitem[Kaveh-Khovanskii10c]{Askold-Kiumars-horo}
Kaveh, K.; Khovanskii, A. G. {\it Newton polytopes for
horospherical spaces}. To appear in {Moscow Mathematical Journal}.

\bibitem[Khovanskii88]{Askold-BZ} Khovanskii, A. G.
{\it Algebra and mixed volumes}. Appendix 3 in: Burago, Yu. D.; Zalgaller, V. A.
{\it Geometric inequalities}. Translated from the Russian by A. B. Sosinskii.
Grundlehren der Mathematischen Wissenschaften, 285.
Springer Series in Soviet Mathematics (1988).

\bibitem[Khovanskii92]{Askold-Hilbert-poly}
Khovanskii, A. G. {\it Newton polyhedron, Hilbert polynomial and sums of finite sets}.
(Russian)  Funktsional. Anal. i Prilozhen.  26  (1992),  no. 4, 57--63, 96;
translation in  Funct. Anal. Appl.  26  (1992),  no. 4, 276--281.

\bibitem[Khovanskii95]{Askold-finite-sets} Khovanskii, A. G. {\it Sums of finite sets, orbits of
commutative semigroups and Hilbert functions}. (Russian)
Funktsional. Anal. i Prilozhen.  29  (1995),  no. 2, 36--50, 95;
translation in  Funct. Anal. Appl.  29  (1995),  no. 2, 102--112.


\bibitem[Kuronya-Lozovanu-Maxlean10]{Kuronya}
Kuronya, A.; Lozovanu, V.; Maclean C.
{\it Convex bodies appearing as Okounkov bodies of divisors}.
{arXiv:1008.4431v1}.

\bibitem[Kushnirenko76]{Kushnirenko} Kushnirenko, A. G.
{\it Polyedres de Newton et nombres de Milnor}. (French)
Invent. Math. 32 (1976), no. 1, 1--31.

\bibitem[Lazarsfeld04]{Lazarsfeld} Lazarsfeld, R.
{\it Positivity in algebraic geometry. I. Classical setting: line bundles and linear series}.
Ergebnisse der Mathematik und ihrer Grenzgebiete. 3.
Folge. A Series of Modern Surveys in Mathematics, 48. Springer-Verlag, Berlin, 2004.

\bibitem[Lazarsfeld-Mustata08]{Lazarsfeld-Mustata}
Lazarsfeld, R.; Mustata, M. {\it Convex bodies associated
to linear series}. Ann. de lÕENS 42 (2009), no. 5, 783-835.

\bibitem[Nystrom09]{Nystrom} Nystrom, D. W. {\it Transforming metrics
on a line bundle to the Okounkov body}. Preprint {arXiv:0903.5167v1}

\bibitem[Okounkov96]{Okounkov-Brunn-Minkowski} Okounkov, A.
{\it Brunn-Minkowski inequality for multiplicities}.
Invent. Math. 125 (1996), no. 3, 405--411.

\bibitem[Okounkov03]{Okounkov-log-concave} Okounkov, A.
{\it Why would multiplicities be log-concave?}
The orbit method in geometry and physics (Marseille, 2000),
329--347, Progr. Math., 213, Birkha"user Boston, Boston, MA, 2003.

\bibitem[Parshin83]{Parshin} Parshin, A. N.
{\it Chern classes, adeles and L-functions}.
J. Reine Angew. Math. 341, 174-192 (1983).

\bibitem[Samuel-Zariski60]{Zariski} Samuel, P.; Zariski, O.
{\it Commutative algebra}. Vol. II. Reprint of the 1960 edition.
Graduate Texts in Mathematics, Vol. 29.


\bibitem[Teissier79]{Teissier} Teissier, B.
{\it Du theoreme de l'index de Hodge aux inegalites isoperimetriques}.
C. R. Acad. Sci. Paris Ser. A-B 288 (1979), no. 4, A287--A289.

\bibitem[Yuan09]{Yuan} Yuan, X. {\it On volumes of arithmetic line bundles}.
{Compos. Math}. 145 (2009), no. 6, 1447Ð-1464. \end{thebibliography}
\end{document}